\documentclass[11pt,a4paper]{article}
\usepackage{graphicx}

\usepackage{amsfonts}
\usepackage{amssymb}
\usepackage{amsmath}
\usepackage{amsthm}
\usepackage{amscd}

\allowdisplaybreaks
\usepackage{graphicx}
\usepackage{pxfonts}

\usepackage{inputenc}

\def\vbar{\mathchoice{\vrule height6.3ptdepth-.5ptwidth.8pt\kern- .8pt}
{\vrule height6.3ptdepth-.5ptwidth.8pt\kern-.8pt} {\vrule
height4.1ptdepth-.35ptwidth.6pt\kern-.6pt} {\vrule
height3.1ptdepth-.25ptwidth.5pt\kern-.5pt}}
\def\fudge{\mathchoice{}{}{\mkern.5mu}{\mkern.8mu}}
\def\bbc#1#2{{\rm \mkern#2mu\vbar\mkern-#2mu#1}}
\def\bbb#1{{\rm I\mkern-3.5mu #1}}
\def\bba#1#2{{\rm #1\mkern-#2mu\fudge #1}}
\def\bb#1{{\count4=`#1 \advance\count4by-64 \ifcase\count4\or\bba
A{11.5}\or \bbb B\or\bbc C{5}\or\bbb D\or\bbb E\or\bbb F \or\bbc
G{5}\or\bbb H\or \bbb I\or\bbc J{3}\or\bbb K\or\bbb L \or\bbb
M\or\bbb N\or\bbc O{5} \or \bbb P\or\bbc C{5}\or\bbb B\or\bbc
S{4.2}\or\bba T{10.5}\or\bbc U{5}\or \bba V{12}\or\bba
W{16.5}\or\bba X{11}\or\bba Y{11.7}\or\bba Z{7.5}\fi}}

\newtheorem{df}{Definition}[section]
\newtheorem{thm}{Theorem}[section]

\newtheorem{rem}{Remark}[section]

\newtheorem{example}[thm]{Example}
\newtheorem{prop}{Proposition}[section]

\begin{document}
\date{}

\title{\bf Poisson superbialgebras}\author{\bf Imed Basdouri,  Mohamed Fadous, Sami Mabrouk, and Abdenacer Makhlouf}
\author{{ Imed Basdouri$^{1}$
 \footnote { E-mail: basdourimed@yahoo.fr}
,\  Mohamed Fadous$^{2}$
    \footnote { E-mail:  mohamedfadous201001@gmail.com}
,\  Sami Mabrouk$^{1}$
 \footnote { E-mail: mabrouksami00@yahoo.fr }
\ and Abdenacer Makhlouf$^{3}$
 \footnote { E-mail: abdenacer.makhlouf@uha.fr $($Corresponding author$)$}
}\\
{\small 1.  University of Gafsa, Faculty of Sciences, 2112 Gafsa, Tunisia} \\
{\small 2. University of Sfax, Faculty of Sciences,  BP
1171, 3038 Sfax, Tunisia }\\
{\small 3.~ Universit\'e de Haute Alsace, IRIMAS - D\'epartement de Math\'ematiques,}\\ {\small 18, rue des fr\`eres Lumi\`ere,
F-68093 Mulhouse, France}}
\date{}
 \maketitle

 \begin{abstract}
We introduce the notion of Poisson superbialgebra as an analogue of  Drinfeld's Lie superbialgebras. We extend various known constructions dealing with representations on Lie superbialgebras to Poisson superbialgebras. We introduce  the notions of Manin triple of Poisson superalgebras and Poisson superbialgebras and show the equivalence between them in terms of  matched pairs of Poisson superalgebras. 
A combination of the classical Yang-Baxter equation and the associative Yang-Baxter equation is discussed in this framework. Moreover, we introduce notions of $\mathcal{O}$-operator of weight $\lambda\in\mathbb{K}$  of a Poisson superalgebra and post-Poisson superalgebra and interpret the close relationships between them and Poisson superbialgebras.
\end{abstract}

\textbf{Keywords}:  Poisson superbialgebra, $\mathcal{O}$-operator, representation, post-Poisson superalgebra.

\textbf{Mathematics Subject Classification} (2020): 17B63, 17B60, 17B70, 17B80.
\maketitle

\section*{Introduction}

A Poisson algebra is both a Lie algebra and a commutative associative algebra which are compatible in a certain sense. Poisson algebras play important roles in many fields in mathematics and mathematical physics, such as the Poisson geometry, integrable systems, non-commutative (algebraic or differential) geometry, and so on (see \cite{Vaisman} and the references therein).\

As generalizations of Lie algebras, Lie superalgebras were introduced motivated by applications to physics and to deformations of Lie algebras, especially Lie algebras of vector fields. The notion of Lie superalgebras was firstly introduced in  \cite{KAC}. More precisely,  Lie superalgebras are $\mathbb{Z}_2$-graded version of a Lie algebras, where  as the Jacobi identity is replaced by a superJacobi
identity given by $$(-1)^{{|x|}{|z|}}[x,[y,z]]
+(-1)^{{|z|}{|y|}}[z,[x,y]]+(-1)^{{|y|}{|x|}}[y,[z,x]]=0.$$
Superization of  Poisson algebras was considered in  classical dynamics of fermion fields and classical spin-1/2 particles, and also   used in the BRST and Batalin-Vilkovisky formalism. It is also related  to Poisson supermanifolds in differential geometry  \cite{Kostant,Kosmann-Schwarzbach0,Kosmann-Schwarzbach}. 

On the other hand, both Lie superalgebras and (commutative) associative superalgebras have a known theory of bialgebras which have been applied to a lot of fields. Lie bialgebras, which were introduced by Drinfeld in the early 1980s for studying the solutions of classical Yang-Baxter equation, are now well established as the infinitesimalisation of quantum groups. In the case
of associative algebras, Joni and Rota introduced the notion of infinitesimal bialgebra in order to provide an algebraic framework for the calculus of divided differences  \cite{Joni-Rota}.  With an additional supersymmetric solutions of the associative Yang-Baxter equation, it is called an associative D-bialgebra, see also   balanced infinitesimal bialgebra  or  antisymmetric infinitesimal bialgebra. All  these
notions are analogues of a Lie bialgebra.   Aguiar developed a systematic theory for infinitesimal bialgebras from this point of view \cite{M. Aguiarrr,M. Aguiar 2,M. Aguiar}. 

The aim of this paper is to study  Poisson superbialgebras from these perspectives and also dealing with matched pairs,  manin triples and $\mathcal{O}$-operators.
The paper is organized as follows. In the first section we provide some basic definitions and in Section  2 we  discuss  a characterization of Poisson superalgebras with one operation and several constructions related to  representations.   In Section 3, we   define various algebraic structures connected to  Poisson superbialgebras. In Section 4, we introduce  the notions of Manin triple of Poisson superalgebras and Poisson superbialgebras and then give the equivalence between them in terms of  matched pairs of Poisson superalgebras. In Section 5, we consider the coboundary cases and discuss Poisson Yang-Baxter equation  (PYBE) which is a combination of Classical Yang-Baxter equation (CYBE) and associative Yang-Baxter equation (AYBE). In Section 6, we introduce notions of $\mathcal{O}$-operator of weight $\lambda\in\mathbb{K}$  of a Poisson superalgebra and post-Poisson superalgebra and interpret the close relationships between them and Poisson superbialgebras.

\section{Definitions and Preliminaries}
First, let us start by fixing some definitions and notations.
 Let $\mathcal{A}=\mathcal{A}_{\bar{0}}\oplus\mathcal{A}_{\bar{1}}$ be a  $\mathbb{Z}_2$-graded vector space over an arbitrary field $\mathbb{K}$ of characteristic 0. In the sequel, we will consider only elements which are  $\mathbb{Z}_{2}$-homogeneous. For $x\in\mathcal{A}$, we denote by ${|x|}\in\mathbb{Z}_{2}$  its parity, i.e., $x\in \mathcal{A}_{|x|}$.
We denote by $\tau$ the super-twist map of $\mathcal{A}\otimes\mathcal{A}$, namely 
$
 \tau(x\otimes y)=(-1)^{{|x|}{|y|}}y\otimes x
$
for $x,y\in \mathcal{A}.$
The super-cyclic map $\xi$ permutes the coordinates of $\mathcal{A}\otimes\mathcal{A}\otimes\mathcal{A}$, it  is defined as
\begin{equation*}
\xi=(id\otimes\tau)\cdot(\tau\otimes id) : x\otimes y\otimes z\mapsto (-1)^{{|x|}({|y|}+{|z|})}y\otimes z\otimes x,
\end{equation*}
for $x, y, z\in\mathcal{A}$, where $id$ is the identity map on $\mathcal{A} $. We denote by $\mathfrak{L^\ast}=$Hom$(\mathcal{A}, \mathbb{K})$ the linear dual of $\mathcal{A}$. For $\phi\in\mathfrak{L^\ast}$ and $x\in\mathcal{A}$, we often use the adjoint notation $\langle\phi, x\rangle$ for $\phi(x)\in\mathbb{K}$.\\
For a linear map $\Delta, \  \Delta: \mathcal{A}\rightarrow \mathcal{A}\otimes\mathcal{A}$ (comultiplication), we use Sweedler's notation $\Delta(x)=\sum_{(x)} x_{1}\otimes x_{2}$ for $x\in\mathcal{A}$.
We will often omit the summation sign $\sum_{(x)}$ to make it simpler.
The parity $|r|$ of $r\in\mathcal{A}^{\otimes2}$ is defined as follows : since we assume $r$  homogenous, there exists $|r|\in\mathbb{Z}_{2}$, such that $r$ can be written as $r=\sum r_{1}\otimes r_{2}\in\mathcal{A}^{\otimes2}$, $r_{1}, r_{2}$ are homogenous elements with $|r|=|r_{1}|+|r_{2}|$.

\begin{df} A \emph{Lie superalgebra} is a pair $(\mathcal{A}, [\cdot ,\cdot ])$ consisting of a superspace $\mathcal{A}$ such that $[\mathcal{A}_{i} ,\mathcal{A}_{j} ]\subset\mathcal{A}_{i+j}$ equipped with a bilinear map $[\cdot ,\cdot ]:\mathcal{A} \times \mathcal{A} \rightarrow \mathcal{A}$ satisfying
\begin{eqnarray}\label{701}
[x,y]=-(-1)^{{|x|}{|y|}}[y,x],
\end{eqnarray}
\begin{eqnarray}\label{702}
(-1)^{{|x|}{|z|}}[x,[y,z]]+(-1)^{{|z|}{|y|}}[z,[x,y]]+(-1)^{{|y|}{|x|}}[y,[z,x]]=0,
\end{eqnarray}
for all homogeneous elements $x,y,z$ in $\mathcal{A}$.\\
\end{df}

\begin{df} A commutative associative superalgebra is a pair $(\mathcal{A}, \mu)$ consisting of a superspace $\mathcal{A}$ an even bilinear map $\mu:\mathcal{A}\times\mathcal{A}\longrightarrow \mathcal{A}$  satisfying
\begin{eqnarray}
\ \ \ \mu(x, y)=(-1)^{{|x|}{|y|}}\mu(y, x),
\end{eqnarray}
\begin{eqnarray}
\mu(x, \mu(y, z))=\mu(\mu(x, y), z) \ \ \ \ \ (associativity),
\end{eqnarray}
for all $x,y$ in $\mathcal{A}$.
\end{df}

\begin{df} A Poisson superalgebra is a triple $(\mathcal{A}, \{,\}, \mu)$ consisting of a superspace $\mathcal{A}$, an even bilinear
map $\{,\} : \mathcal{A}\times \mathcal{A} \longrightarrow \mathcal{A}$ and $\mu: \mathcal{A}\times \mathcal{A} \longrightarrow \mathcal{A}$, satisfying
\begin{enumerate}
\item $(\mathcal{A}, \{,\})$ is a Lie superalgebra,
\item $(\mathcal{A}, \mu)$ is a commutative associative superalgebra,
\item for all $x, y \in\mathcal{A}$ :
\begin{eqnarray}\label{awww}
\{x,\mu(y, z)\} = \mu(\{x, y\},z)+(-1)^{{|x|}{|y|}}\mu(y, \{x,z\}).
\end{eqnarray}
\end{enumerate}
\end{df}

The condition (\ref{awww}) expresses the compatibility between the Lie bracket $\{,\}$ and the associative superalgebra product $\mu$, it can be written equivalently as
\begin{eqnarray}
\{\mu(x, y), z\} = \mu(x, \{y, z\})+(-1)^{{|x|}{|y|}}\mu(y, \{x,z\}).
\end{eqnarray}

A homomorphism between two Poisson superalgebras is defined as a linear map between two Poisson superalgebras preserving the corresponding operations.

\section{Modules and matched pairs of Poisson superalgebras}
First, we recall how to construct a Lie superalgebra (or a commutative associative superalgebra) structure on the direct sum of two Lie superalgebra (or two commutative associative superalgebra) such that both of them are sub-superalgebra.\\

Let $\cdot:(x, y)\longrightarrow x\cdot y$ be an even bilinear map on the superspace $\mathcal{A}$. The associator $as$ of $\cdot$ is the trilinear map on $\mathcal{A}$ given by $$as(x, y, z)=(x\cdot y)\cdot z-x\cdot(y\cdot z).$$

\begin{prop} Let $(\mathcal{A}, \cdot)$ be a $\mathbb{K}$-superalgebra. Define the $\mathcal{A}$-valued operations $\{,\}$ and $\bullet$ on $\mathcal{A}\times\mathcal{A}$ by
\begin{eqnarray}
\{x, y\}=\frac{1}{2}(x\cdot y-(-1)^{{|x|}{|y|}}y\cdot x),
\end{eqnarray}
\begin{eqnarray}
x\bullet y=\frac{1}{2}(x\cdot y+(-1)^{{|x|}{|y|}}y\cdot x).
\end{eqnarray}
Then $(\mathcal{A}, \{,\}, \bullet)$ is a Poisson superalgebra if and only if the operation $x\cdot y$ satisfies the identity: $$ 3\ as(x, y, z)=(-1)^{{|y|}{|z|}}(x\cdot z)\cdot y-(-1)^{{|x|}{|z|}+{|y|}{|z|}}(z\cdot x)\cdot y+(-1)^{{|x|}{|y|}+{|x|}{|z|}}(y\cdot z)\cdot x-(-1)^{{|x|}{|y|}}(y\cdot x)\cdot z.$$
\end{prop}
\begin{proof} For any $x, y, z\in\mathcal{A}$, we will check that the $(\mathcal{A}, \{,\})$ is a Lie superalgebra. We have
\begin{eqnarray*}
 && (-1)^{{|x|}{|z|}}\{x, \{y, z\}\}+(-1)^{{|y|}{|x|}}\{y, \{z, x\}\}+(-1)^{{|z|}{|y|}}\{z, \{x, y\}\} 
 \\ && 
 =(-1)^{{|x|}{|y|}}\{x, \frac{1}{2}(y\cdot z-(-1)^{{|y|}{|z|}}z\cdot y)\}+
(-1)^{{|y|}{|x|}}\{y, \frac{1}{2}(z\cdot x-(-1)^{{|z|}{|x|}}z\cdot x)\}
\\ &&  \quad 
+ (-1)^{{|z|}{|y|}}\{z, \frac{1}{2}(x\cdot y-(-1)^{{|x|}{|y|}}y\cdot x)\} 
\\ && 
 =\frac{1}{4}\Big(-(-1)^{{|x|}{|z|}}as(x, y, z)+(-1)^{{|x|}{|z|}+{|y|}{|z|}}as(x, z, y)-(-1)^{{|x|}{|y|}}as(y, z, x)+(-1)^{{|x|}{|y|}+{|y|}{|z|}}as(z, y, x)
 \\ && \quad 
 +(-1)^{{|x|}{|y|}+{|x|}{|z|}}as(y, x, z)-(-1)^{{|y|}{|z|}}as(z, x, y) \Big)=0.
\end{eqnarray*}
Next, we check that $(\mathcal{A}, \bullet)$ is a commutative associative superalgebra. Indeed, for any $x, y, z\in\mathcal{A}$,
\begin{align*}
&(x\bullet y)\bullet z-x\bullet(y\bullet z)=\frac{1}{2}(x\cdot y+(-1)^{{|x|}{|y|}}y\cdot x)\bullet z
-x\bullet\frac{1}{2}(y\cdot z+(-1)^{{|y|}{|z|}}y\cdot z)\\
&=\frac{1}{4}\Big(as(x, y, z)-(-1)^{{|x|}{|y|}+{|x|}{|z|}+{|y|}{|z|}}as(z, y, x)+(-1)^{{|x|}{|y|}}(y\cdot x)\cdot z-(-1)^{{|x|}{|y|}+{|x|}{|z|}}(y\cdot z)\cdot x
\\ & \quad -(-1)^{{|x|}{|z|}+{|y|}{|z|}}as(z, x, y)
+(-1)^{{|x|}{|z|}+{|y|}{|z|}}(z\cdot x)\cdot y+(-1)^{{|y|}{|z|}}as(x, z, y)-(-1)^{{|y|}{|z|}}(x\cdot z)\cdot y\Big)=0.
\end{align*}
Finally, we check the condition :\ \ \
$\{x, y\bullet z\}=\{x, y\}\bullet z+(-1)^{{|x|}{|y|}}y\bullet\{x, z\}.$  We have
\begin{align*}
&\{x, y\bullet z\}-\{x, y\}\bullet z-(-1)^{{|x|}{|y|}}y\bullet\{x, z\}\\
&=\frac{1}{4}\Big(-as(x, y, z)-(-1)^{{|y|}{|z|}}as(x, z, y)-(-1)^{{|x|}{|y|}+{|x|}{|z|}}as(y, z, x)-(-1)^{{|x|}{|y|}+{|x|}{|z|}+{|y|}{|z|}}as(z, y, x)\\
&+(-1)^{{|x|}{|y|}}as(y, x, z)+(-1)^{{|x|}{|z|}+{|y|}{|z|}}as(z, x, y)\Big)=0.
\end{align*}
The proof is finished.
\end{proof}

\begin{df} Let $(\mathcal{A}, [\cdot ,\cdot ])$ be a Lie superalgebra and $V=V_{\bar{0}}\oplus V_{\bar{1}}$ an arbitrary vector superspace. A  representation (or module) of the Lie superalgebra  is given by a pair $(V, \rho)$, where  $\rho:\mathcal{A} \rightarrow End(V)$ is an  even linear map such that $\rho(\mathcal{A}_{i})(V_{j})\subset V_{i+j}$ where $i, j\in\mathbb{Z}_{2}$, and satisfying
\begin{eqnarray}
\rho([x, y])=\rho(x)\circ\rho(y)-(-1)^{{|x|}{|y|}}\rho(y)\circ\rho(x).
\end{eqnarray}
for all homogeneous elements $x, y \in\mathcal{A}$.
\end{df}

Let $(\mathcal{A},[\cdot,\cdot]_{\mathcal{A}})$ and $(\mathcal{A}',[\cdot,\cdot]_{\mathcal{A}'})$ be two Lie superalgebras. Set $\mathcal{A}=\mathcal{A}_{\bar{0}}\oplus\mathcal{A}_{\bar{1}}$ and $\mathcal{A}'=\mathcal{A}'_{\bar{0}}\oplus\mathcal{A}'_{\bar{1}}$. Let $\rho:\mathcal{A}\longrightarrow End(\mathcal{A}')$ and $\rho':\mathcal{A}'\longrightarrow End(\mathcal{A})$ be two even linear maps. Define a skew-supersymmetric bracket $[\cdot ,\cdot ]: \widetilde{G}\times\widetilde{G}\rightarrow \widetilde{G}$, where $\widetilde{G}$ is given by : $\widetilde{G}=\widetilde{G}_{\bar{0}}\oplus \widetilde{G}_{\bar{1}}$, with $\widetilde{G}_{\bar{0}}=\mathcal{A}_{\bar{0}}\oplus\mathcal{A}'_{\bar{0}}$ and $\widetilde{G}_{\bar{1}}=\mathcal{A}_{\bar{1}}\oplus\mathcal{A}'_{\bar{1}}$. We set  $|x+x'|=|x|=|x'|$ and $|y+y'|=|y|=|y'|$, for all homogeneous elements $x, y$  in $\mathcal{A}$ and $x', y'$ in  $\mathcal{A}'$.
On the direct sum of the underlying vector superspace $\mathcal{A}\oplus \mathcal{A}'$, the bracket,  on $\widetilde{G}$,  
 $[\cdot,\cdot]:\mathcal{A}\oplus \mathcal{A}' \times \mathcal{A}\oplus \mathcal{A}'\longrightarrow \mathcal{A}\oplus \mathcal{A}'$ is defined  by
\begin{eqnarray}
 \ \ \ \ \ \ \ \ \ \ [x+x',y+y']= [x,y]_\mathcal{A}+\rho'(x')(y)-(-1)^{{|x|}{|y|}}\rho'(y')(x)+[x',y']_{\mathcal{A}'}+\rho(x)(y')-(-1)^{{|x|}{|y|}}\rho(y)(x').
\end{eqnarray}

\begin{thm} \cite{newapproach-Bai} With the above notations, $(\mathcal{A}\oplus \mathcal{A}',[\cdot,\cdot])$ is a Lie superalgebra if and only if $\rho$ and $\rho'$ are representations of $\mathcal{A}$ and $\mathcal{A}'$ respectively and the following
conditions are satisfied
\begin{align}\label{seff}
\rho'(z')[x,y]_\mathcal{A} &= [\rho'(z')(x), y]_\mathcal{A}+(-1)^{{|x|}{|z|}}[x,\rho'(z')(y)]_\mathcal{A}\\
 &+ (-1)^{{{|x|}{|y|}}+{{|y|}{|z|}}} \rho'(\rho(y)(z'))(x)-(-1)^{{|x|}{|z|}}\rho'(\rho(x)(z'))(y),
 \nonumber\end{align}
\begin{align}\label{sff}
\rho(z)[x',y']_{\mathcal{A}'} &= [\rho(z)(x'), y']_{\mathcal{A}'}+(-1)^{{|x|}{|z|}}[x',\rho(z)(y')]_{\mathcal{A}'}\\
 &+ (-1)^{{{|x|}{|y|}}+{{|y|}{|z|}}}\rho(\rho'(y')(z))(x')-(-1)^{{|x|}{|z|}}\rho(\rho'(x')(z))(y'),
 \nonumber\end{align}
for any $x, y, z\in\mathcal{A}$ and $x',y', z'\in\mathcal{A}'$.
\end{thm}

This Lie superalgebra  is denoted by $\mathcal{A} \bowtie^{\rho'}_{\rho}\mathcal{A}'$ or simply $\mathcal{A}\bowtie\mathcal{A}'$. Moreover, every Lie superalgebra which is the direct sum of the underlying vector superspaces of two sub-superalgebra can be obtained from a matched pair of Lie superalgebras.

\begin{df}
Let $(\mathcal{A},\bullet)$ be a commutative associative superalgebra. Recall that a representation (or module) on a vector superspace $V=V_{\bar{0}}\oplus V_{\bar{1}}$ with respect to $\mathcal{A}$ is an   even linear map $\varphi:\mathcal{A}\longrightarrow End(V)$, such that for any $x, y\in\mathcal{A}$, the following equality is satisfied:
\begin{eqnarray}
\varphi(x\bullet y)=\varphi(x)\circ\varphi(y).
\end{eqnarray}
\end{df}
We denote the representation by $(V, \varphi)$.

\

Let $(\mathcal{A}_1,\bullet_1)$ and $(\mathcal{A}_2,\bullet_2)$ be two commutative associative superalgebras. Set $\mathcal{A}_1={\mathcal{A}_1}_{\bar{0}}\oplus{\mathcal{A}_1}_{\bar{1}}$ and $\mathcal{A}_2={\mathcal{A}_2}_{\bar{0}}\oplus{\mathcal{A}_2}_{\bar{1}}$. If there are  linear maps $\varphi_1:\mathcal{A}_1\longrightarrow End(\mathcal{A}_2)$ and $\varphi_2:\mathcal{A}_2\longrightarrow End(\mathcal{A}_1)$ which are representations of $\mathcal{A}_1$ and $\mathcal{A}_2$. Define a bilinear map $\bullet: \widetilde{G}\times\widetilde{G}\rightarrow \widetilde{G}$, where $\widetilde{G}$ is given by  $\widetilde{G}=\widetilde{G}_{\bar{0}}\oplus \widetilde{G}_{\bar{1}}$, with $\widetilde{G}_{\bar{0}}={\mathcal{A}_1}_{\bar{0}}\oplus{\mathcal{A}_2}_{\bar{0}}$ and $\widetilde{G}_{\bar{1}}={\mathcal{A}_1}_{\bar{1}}\oplus{\mathcal{A}_2}_{\bar{1}}$. We set  $|x+a|=|x|=|a|$ and $|y+b|=|y|=|b|$, for all $x, y$ in $\mathcal{A}_1$ and $a, b$ in  $\mathcal{A}_2$.
On the direct sum of the underlying vector superspace $\mathcal{A}_1\oplus \mathcal{A}_2$, define a bilinear map $\bullet$ by
\begin{eqnarray}
(x+a)\bullet (y+b)&=& x\bullet_1 y+\varphi_2(a)y+(-1)^{{|x|}{|y|}}\varphi_2(b)x+ a\bullet_2 b+\varphi_1(x)b+(-1)^{{|x|}{|y|}}\varphi_1(y)a.
\end{eqnarray}

\begin{thm} With the above notations, $(\mathcal{A}_1\oplus \mathcal{A}_2,\bullet)$ is a commutative associative superalgebra if and only if $\varphi_1$ and $\varphi_2$ are representations of $\mathcal{A}_1$ and $\mathcal{A}_2$ respectively and the following conditions are satisfied
\begin{eqnarray}\label{01}
\varphi_1(x)(a \bullet_2 b) &=& (\varphi_1(x)a)\bullet_2 b+ (-1)^{{|x|}{|a|}}\varphi_1(\varphi_2(a)x)b,
\end{eqnarray}
\begin{eqnarray}\label{02}
\varphi_2(a)(x \bullet_1 y) &=& (\varphi_2(a)x)\bullet_1 y+ (-1)^{{|x|}{|a|}}\varphi_2(\varphi_1(x)a) y,
\end{eqnarray}
for any $x,y\in\mathcal{A}_1$ and $a,b\in\mathcal{A}_2$.
\end{thm}

\begin{proof} If  $(\mathcal{A},\bullet)$ is a commutative associative superalgebra, with respect to $\bullet$ we have
\begin{eqnarray}\label{198a}
\textbf{(}x+a\textbf{)}\bullet\textbf{(}(y+b)\bullet(z+c)\textbf{)}=\textbf{(}(x+a)\bullet(y+b)\textbf{)}\bullet \textbf{(}z+c\textbf{)},
\end{eqnarray}
for $x, y, z\in\mathcal{A}_1$ and $a, b, c\in\mathcal{A}_2$. Expanding (\ref{198a}), we obtain
\begin{align*}
\ast) \ &\textbf{(}x+a\textbf{)}\bullet\textbf{(}(y+b)\bullet(z+c)\textbf{)}=
x\bullet_1(y\bullet_1z)+x\bullet_1(\varphi_2(b)z)+(-1)^{{|y|}{|z|}}x\bullet_1(\varphi_2(c)y)+\varphi_2(a)(y\bullet_1z)\\
&+\varphi_2(a)(\varphi_2(b)z)+(-1)^{{|y|}{|z|}}\varphi_2(a)(\varphi_2(c)y)
+(-1)^{{{|x|}{|y|}}+{{|x|}{|z|}}}\varphi_2(b\bullet_2c)x+(-1)^{{{|x|}{|y|}}+{{|x|}{|z|}}}\varphi_2(\varphi_1(y)c)x\\
&+(-1)^{{{|x|}{|y|}}+{{|x|}{|z|}}+{{|y|}{|z|}}}\varphi_2(\varphi_1(z)b)x+a\bullet_2(b\bullet_2c)+a\bullet_2(\varphi_1(y)c)+(-1)^{{|y|}{|z|}}a\bullet_2(\varphi_1(z)b)
+\varphi_1(x)(b\bullet_2c)\\
&+\varphi_1(x)(\varphi_1(y)c)+(-1)^{{|y|}{|z|}}\varphi_1(x)(\varphi_1(z)b)
+(-1)^{{{|x|}{|y|}}+{{|x|}{|z|}}}\varphi_1(y\bullet_1z)a+(-1)^{{{|x|}{|y|}}+{{|x|}{|z|}}}\varphi_1(\varphi_2(b)z)a\\
&+(-1)^{{{|x|}{|y|}}+{{|x|}{|z|}}+{{|y|}{|z|}}}\varphi_1(\varphi_2(c)y)a.
\end{align*}
\begin{align*}
\ast) \ &\textbf{(}(x+a)\bullet(y+b)\textbf{)}\bullet \textbf{(}z+c\textbf{)}=
(x\bullet_1y)\bullet_1z+\varphi_2(a)y\bullet_1z+(-1)^{{|x|}{|y|}}\varphi_2(b)x\bullet_1z
+\varphi_2(a\bullet_2b)z+\varphi_2(\varphi_1(x)b)z\\
&+(-1)^{{|x|}{|y|}}\varphi_2(\varphi_1(y)a)z
+(-1)^{{{|z|}{|x|}}+{{|z|}{|y|}}}\varphi_2(c)(x\bullet_1y)
+(-1)^{{{|z|}{|x|}}+{{|z|}{|y|}}}\varphi_2(c)\varphi_2(a)y
+(-1)^{{{|z|}{|x|}}+{{|z|}{|y|}}}\varphi_1(z)\varphi_1(x)b\\
&+(a\bullet_2b)\bullet_2c+\varphi_1(x)b\bullet_2c+(-1)^{{|x|}{|y|}}\varphi_1(y)a\bullet_2c
+\varphi_1(x\bullet_1y)c+\varphi_1(\varphi_2(a)y)c
+(-1)^{{|x|}{|y|}}\varphi_1(\varphi_2(b)x)c\\
&+(-1)^{{{|z|}{|x|}}+{{|z|}{|y|}}}\varphi_1(z)(a\bullet_2b)
+(-1)^{{{|z|}{|x|}}+{{|z|}{|y|}}+{{|x|}{|y|}}}\varphi_1(z)\varphi_1(y)a
+(-1)^{{{|z|}{|x|}}+{{|z|}{|y|}}+{{|x|}{|y|}}}\varphi_2(c)\varphi_2(b)x.\\
\end{align*}

It implies (\ref{01}), (\ref{02}) and
\begin{eqnarray}\label{rep 1}
\varphi_1(x\bullet_1y)c=\varphi_1(x)\varphi_1(y)c,
\end{eqnarray}
\begin{eqnarray}\label{rep 2}
\varphi_2(a\bullet_2b)z=\varphi_2(a)\varphi_2(b)z.
\end{eqnarray}
By (\ref{rep 1}), we deduce that $\varphi_1$ is a representation of $(\mathcal{A}_1,\bullet_1)$ on $\mathcal{A}_2$. By (\ref{rep 2}) we deduce that $\varphi_2$ is a representation of $(\mathcal{A}_2,\bullet_2)$ on $\mathcal{A}_1$. This finishes the proof.
\end{proof}

We use the notation  $\mathcal{A}_1 \bowtie^{\varphi_2}_{\varphi_1}\mathcal{A}_2$ or simply $\mathcal{A}_1\bowtie\mathcal{A}_2$. Moreover, every commutative associative superalgebra which is the direct sum of the underlying vector superspaces of two sub-superalgebra can be obtained from a matched pair of commutative associative superalgebras.

\begin{df}
Let $(\mathcal{P},\{,\},\bullet)$ be a Poisson superalgebra, $V=V_{\bar{0}}\oplus V_{\bar{1}}$ an arbitrary vector superspace and $ \psi_{\{,\}}, \psi_{\bullet}:\mathcal{P}\longrightarrow End(V)$ be two even linear maps. Then $(V, \psi_{\{,\}}, \psi_{\bullet})$ is called a representation (or module) of $\mathcal{P}$ if $(V,\psi_{\{,\}})$ is a representation of $(\mathcal{P},\{,\})$ and $(V, \psi_{\bullet})$ is a representation of $(\mathcal{P},\bullet)$ and they are compatible in the sense that for any $x,y\in \mathcal{P}$ and $v\in V$
\begin{eqnarray}\label{super rep 1}
\psi_{\{,\}}(x\bullet y)(v)= \psi_{\bullet}(x)\psi_{\{,\}}(y)(v)+(-1)^{{|x|}{|y|}}\psi_{\bullet}(y)\psi_{\{,\}}(x)(v)
\end{eqnarray}
\begin{eqnarray}\label{super rep 2}
\psi_{\bullet}(\{x,y\})(v) = \psi_{\{,\}}(x)\psi_{\bullet}(y)(v)-(-1)^{{|x|}{|y|}}\psi_{\bullet}(y)\psi_{\{,\}}(x)(v).
\end{eqnarray}
\end{df}

In the case of Poisson superalgebras, we can construct semidirect product when given bimodules. Analogously, we have

\begin{prop}\label{S()YYY} Let $(\mathcal{P},\{,\},\bullet)$ be a Poisson superalgebra. Then $(V, \psi_{\{,\}}, \psi_{\bullet})$ is a representation of a Poisson superalgebra $(\mathcal{P},\{,\},\bullet)$ if and only if the direct sum of vector
superspaces $\mathcal{P}\oplus V$ is turned into a Poisson superalgebra by defining the operations by (we still denote the operations by $\{,\}$ and $\bullet$):
\begin{eqnarray}
\{x_{1}+v_{1}, x_{2}+v_{2}\}= \{x_{1}, x_{2}\}+\psi_{\{,\}}(x_{1})v_{2}-(-1)^{{|x_{1}|}{|x_{2}|}}\psi_{\{,\}}(x_{2})v_{1},
\end{eqnarray}
\begin{eqnarray}
(x_{1}+v_{1})\bullet(x_{2}+v_{2})=x_{1}\bullet x_{2}+\psi_{\bullet}(x_{1})v_{2}+(-1)^{{|x_{1}|}{|x_{2}|}}\psi_{\bullet}(x_{2})v_{1},
\end{eqnarray}
for any $x_{1}, x_{2}\in\mathcal{P}$, $v_{1}, v_{2}\in V$. We denote it by $\mathcal{P}\ltimes_{\psi_{\{,\}},\psi_{\bullet}}V$ or simply $\mathcal{P}\ltimes V$  \cite{R. Schafer}.
\end{prop}

\begin{proof} The sufficient condition holds obviously. Here we just verify the necessary condition. For any $x_{1}, x_{2}, x_{3}\in\mathcal{P}$, $v_{1}, v_{2}, v_{3}\in V$. From the construction above we have $|x_{1}+v_{1}|=|x_{1}|=|v_{1}|$ and $|x_{2}+v_{2}|=|x_{2}|=|v_{2}|$,
\begin{align*}
\{x_{1}+v_{1}, x_{2}+v_{2}\}&=\{x_{1}, x_{2}\}+\psi_{\{,\}}(x_{1})v_{2}-(-1)^{{|x_{1}|}{|x_{2}|}}\psi_{\{,\}}(x_{2})v_{1}\\
&=-(-1)^{{|x_{1}|}{|x_{2}|}}\textbf{(}\{x_{2},x_{1}\}+\psi_{\{,\}}(x_{2})v_{1}-(-1)^{{|x_{1}|}{|x_{2}|}}\psi_{\{,\}}(x_{1})v_{2}\textbf{)}\\
&=-(-1)^{{|x_{1}|}{|x_{2}|}}\{x_{2}+v_{2}, x_{1}+v_{1}\}\\
&=-(-1)^{{|x_{1}+v_{1}|}{|x_{2}+v_{2}|}}\{x_{2}+v_{2}, x_{1}+v_{1}\},
\end{align*}
so the skew-supersymmetry holds. For the superJacobi identity, we have
\begin{align*}
&(-1)^{{|(x_{1}+v_{1})|}{|(x_{3}+v_{3})|}}\{x_{1}+v_{1}, \{x_{2}+v_{2}, x_{3}+v_{3}\}\}+(-1)^{{|(x_{2}+v_{2})|}{|(x_{1}+v_{1})|}}\{x_{2}+v_{2}, \{x_{3}+v_{3}, x_{1}+v_{1}\}\}\\
&+(-1)^{{|(x_{2}+v_{2})|}{|(x_{3}+v_{3})|}}\{x_{3}+v_{3}, \{x_{1}+v_{1}, x_{2}+v_{2}\}\}\\
&=(-1)^{{|x_{1}|}{|x_{3}|}}\psi_{\{,\}}(x_{1})\psi_{\{,\}}(x_{2})v_{3}
-(-1)^{{{|x_{1}|}{|x_{3}|}}+{{|x_{2}|}{|x_{3}|}}}\psi_{\{,\}}(x_{1})\psi_{\{,\}}(x_{3})v_{2}
-(-1)^{{|x_{1}|}{|x_{2}|}}\psi_{\{,\}}(x_{2})\psi_{\{,\}}(x_{3})v_{1}\\
&+(-1)^{{{|x_{1}|}{|x_{2}|}}+{{|x_{2}|}{|x_{3}|}}}\psi_{\{,\}}(x_{3})\psi_{\{,\}}(x_{2})v_{1}
+(-1)^{{|x_{1}|}{|x_{2}|}}\psi_{\{,\}}(x_{2})\psi_{\{,\}}(x_{3})v_{1}
-(-1)^{{{|x_{1}|}{|x_{2}|}}+{{|x_{1}|}{|x_{3}|}}}\psi_{\{,\}}(x_{2})\psi_{\{,\}}(x_{1})v_{3}\\
&-(-1)^{{|x_{2}|}{|x_{3}|}}\psi_{\{,\}}(x_{3})\psi_{\{,\}}(x_{1})v_{2}
+(-1)^{{{|x_{2}|}{|x_{3}|}}+{{|x_{1}|}{|x_{3}|}}}\psi_{\{,\}}(x_{1})\psi_{\{,\}}(x_{3})v_{2}
+(-1)^{{|x_{2}|}{|x_{3}|}}\psi_{\{,\}}(x_{3})\psi_{\{,\}}(x_{1})v_{2}\\
&-(-1)^{{{|x_{2}|}{|x_{3}|}}+{{|x_{1}|}{|x_{2}|}}}\psi_{\{,\}}(x_{3})\psi_{\{,\}}(x_{2})v_{1}
-(-1)^{{|x_{1}|}{|x_{3}|}}\psi_{\{,\}}(x_{1})\psi_{\{,\}}(x_{2})v_{3}
+(-1)^{{{|x_{1}|}{|x_{3}|}}+{{|x_{1}|}{|x_{2}|}}}\psi_{\{,\}}(x_{2})\psi_{\{,\}}(x_{1})v_{3}\\
&=0.
\end{align*}
So $(\mathcal{P}\ltimes V, \{,\})$ is a Lie superalgebra. Next we will check that $(\mathcal{P}\ltimes V, \bullet)$ is a commutative associative superalgebra. In fact, for any $x_{1}, x_{2}, x_{3}\in\mathcal{P}$, $v_{1}, v_{2}, v_{3}\in V$, also one may check directly that :
$$(x_{1}+v_{1})\bullet(x_{2}+v_{2})
=(-1)^{{|x_{1}|}{|x_{2}|}}(x_{2}+v_{2})\bullet(x_{1}+v_{1})
=(-1)^{{(|x_{1}+v_{1}|)}{(|x_{2}+v_{2}|)}}(x_{2}+v_{2})\bullet(x_{1}+v_{1}).$$
Now we check the following equality
$$\textbf{(}x_{1}+v_{1}\textbf{)}\bullet\textbf{(}(x_{2}+v_{2})\bullet(x_{3}+v_{3})\textbf{)}
=\textbf{(}(x_{1}+v_{1})\bullet(x_{2}+v_{2})\textbf{)}\bullet \textbf{(}x_{3}+v_{3}\textbf{)}.$$
For this, we calculate
\begin{align*}
&\textbf{(}x_{1}+v_{1}\textbf{)}\bullet\textbf{(}(x_{2}+v_{2})\bullet(x_{3}+v_{3})\textbf{)}
=\textbf{(}x_{1}+v_{1}\textbf{)}\bullet\textbf{(}x_{2}\bullet x_{3}+\psi_{\bullet}(x_{2})v_{3}+(-1)^{{|x_{2}|}{|x_{3}|}}\psi_{\bullet}(x_{3})v_{2}\textbf{)}\\
&=x_{1}\bullet(x_{2}\bullet x_{3})+\psi_{\bullet}(x_{1})\psi_{\bullet}(x_{2})v_{3}+(-1)^{{|x_{2}|}{|x_{3}|}}\psi_{\bullet}(x_{1})\psi_{\bullet}(x_{3})v_{2}
+(-1)^{{{|x_{1}|}{|x_{2}|}}+{{|x_{1}|}{|x_{3}|}}}\psi_{\bullet}(x_{2}\bullet x_{3})v_{1}\\
&=(x_{1}\bullet x_{2})\bullet x_{3}+\psi_{\bullet}(x_{1}\bullet x_{2})v_{3}+(-1)^{{{|x_{2}|}{|x_{3}|}}+{{|x_{1}|}{|x_{3}|}}}\psi_{\bullet}(x_{3})\psi_{\bullet}(x_{1})v_{2}
+(-1)^{{{|x_{1}|}{|x_{2}|}}+{{|x_{1}|}{|x_{3}|}}+{{|x_{2}|}{|x_{3}|}}}\psi_{\bullet}(x_{3})\psi_{\bullet}(x_{2})v_{1}\\
&=\textbf{(}(x_{1}+v_{1})\bullet(x_{2}+v_{2})\textbf{)}\bullet \textbf{(}x_{3}+v_{3}\textbf{)}.
\end{align*}
Finally, we check the condition
$$\{\textbf{(}x_{1}+v_{1}\textbf{)},\textbf{(}(x_{2}+v_{2})\bullet(x_{3}+v_{3})\textbf{)}\}
=\{x_{1}+v_{1}, x_{2}+v_{2}\}\bullet \textbf{(}x_{3}+v_{3}\textbf{)}
+(-1)^{{|(x_{1}+v_{1})|}{|(x_{2}+v_{2})|}}\textbf{(}x_{2}+v_{2}\textbf{)}\bullet \{x_{1}+v_{1}, x_{3}+v_{3}\}.$$
In fact, on the one hand, we have
\begin{align*}
&\{\textbf{(}x_{1}+v_{1}\textbf{)},\textbf{(}(x_{2}+v_{2})\bullet(x_{3}+v_{3})\textbf{)}\}
=\{\textbf{(}x_{1}+v_{1}\textbf{)}, x_{2}\bullet x_{3}+\psi_{\bullet}(x_{2})v_{3}+(-1)^{{|x_{2}|}{|x_{3}|}}\psi_{\bullet}(x_{3})v_{2}\}\\
&=\{x_{1}, x_{2}\}\bullet x_{3}+\psi_{\{,\}}(x_{1})\psi_{\bullet}(x_{2})v_{3}
+(-1)^{{|x_{2}|}{|x_{3}|}}\psi_{\{,\}}(x_{1})\psi_{\bullet}(x_{3})v_{2}
-(-1)^{{{|x_{1}|}{|x_{2}|}}+{{|x_{1}|}{|x_{3}|}}}\psi_{\bullet}(x_{2})\psi_{\{,\}}(x_{3})v_{1}\\
&-(-1)^{{{|x_{1}|}{|x_{2}|}}+{{|x_{1}|}{|x_{3}|}}+{{|x_{2}|}{|x_{3}|}}}\psi_{\bullet}(x_{3})\psi_{\{,\}}(x_{2})v_{1}.
\end{align*}
On the other hand, we have
\begin{align*}
&\{x_{1}+v_{1}, x_{2}+v_{2}\}\bullet \textbf{(}x_{3}+v_{3}\textbf{)}
+(-1)^{{|(x_{1}+v_{1})|}{|(x_{2}+v_{2})|}}\textbf{(}x_{2}+v_{2}\textbf{)}\bullet \{x_{1}+v_{1}, x_{3}+v_{3}\}\\
&=\textbf{(}\{x_{1}, x_{2}\}+\psi_{\{,\}}(x_{1})v_{2}-(-1)^{{|x_{1}|}{|x_{2}|}}\psi_{\{,\}}(x_{2})v_{1}\textbf{)}\bullet \textbf{(}x_{3}+v_{3}\textbf{)}\\
&+(-1)^{{|x_{1}|}{|x_{2}|}}\textbf{(}x_{2}+v_{2}\textbf{)}\bullet \textbf{(}\{x_{1}, x_{3}\}+\psi_{\{,\}}(x_{1})v_{3}-(-1)^{{|x_{1}|}{|x_{3}|}}\psi_{\{,\}}(x_{3})v_{1}\textbf{)}\\
&=\{x_{1}, x_{2}\}\bullet x_{3}+\psi_{\{,\}}(x_{1})\psi_{\bullet}(x_{2})v_{3}
-(-1)^{{|x_{1}|}{|x_{2}|}}\psi_{\bullet}(x_{2})\psi_{\{,\}}(x_{1})v_{3}
+(-1)^{{{|x_{1}|}{|x_{3}|}}+{{|x_{2}|}{|x_{3}|}}}\psi_{\bullet}(x_{3})\psi_{\{,\}}(x_{1})v_{2}\\
&-(-1)^{{{|x_{1}|}{|x_{3}|}}+{{|x_{2}|}{|x_{3}|}}+{{|x_{1}|}{|x_{2}|}}}\psi_{\bullet}(x_{3})\psi_{\{,\}}(x_{2})v_{1}
+(-1)^{{|x_{1}|}{|x_{2}|}}x_{2}\bullet\{x_{1}, x_{3}\}
+(-1)^{{|x_{1}|}{|x_{2}|}}\psi_{\bullet}(x_{2})\psi_{\{,\}}(x_{1})v_{3}\\
&-(-1)^{{{|x_{1}|}{|x_{2}|}}+{{|x_{1}|}{|x_{3}|}}}\psi_{\bullet}(x_{2})\psi_{\{,\}}(x_{3})v_{1}
+(-1)^{{|x_{2}|}{|x_{3}|}}\psi_{\{,\}}(x_{1})\psi_{\bullet}(x_{3})v_{2}
-(-1)^{{{|x_{2}|}{|x_{3}|}}+{{|x_{1}|}{|x_{3}|}}}\psi_{\bullet}(x_{3})\psi_{\{,\}}(x_{1})v_{2}\\
&=\{x_{1}, x_{2}\}\bullet x_{3}+\psi_{\{,\}}(x_{1})\psi_{\bullet}(x_{2})v_{3}
+(-1)^{{|x_{2}|}{|x_{3}|}}\psi_{\{,\}}(x_{1})\psi_{\bullet}(x_{3})v_{2}
-(-1)^{{{|x_{1}|}{|x_{2}|}}+{{|x_{1}|}{|x_{3}|}}}\psi_{\bullet}(x_{2})\psi_{\{,\}}(x_{3})v_{1}\\
&-(-1)^{{{|x_{1}|}{|x_{2}|}}+{{|x_{1}|}{|x_{3}|}}+{{|x_{2}|}{|x_{3}|}}}\psi_{\bullet}(x_{3})\psi_{\{,\}}(x_{2})v_{1}.
\end{align*}
Thus $(\mathcal{P}\ltimes V, \{,\}, \bullet)$ is a Poisson superalgebra. The proof is completed.
\end{proof}

\begin{thm} Let $(\mathcal{P}, \{,\}, \bullet)$ be a Poisson superalgebra and $(V, \psi_{\{,\}}, \psi_{\bullet})$ be  a representation of a Poisson superalgebra $(\mathcal{P}, \{,\}, \bullet)$. Define $\psi^{\ast}_{\{,\}}, \psi^{\ast}_{\bullet}:\mathcal{P} \longrightarrow End(V^{\ast})$ by, for any x $\in\mathcal{P}$, $\xi\in V^{\ast}$, $v\in V$
\begin{eqnarray}
\langle \psi^{\ast}_{\{,\}}(x)\xi,  v\rangle=-(-1)^{{|x|}{|\xi|}}\langle \xi,  \psi_{\{,\}}(x)v\rangle, \ \ \ \ \ \ \langle \psi^{\ast}_{\bullet}(x)\xi,  v\rangle=-(-1)^{{|x|}{|\xi|}}\langle \xi,  \psi_{\bullet}(x)v\rangle.
\end{eqnarray}
If, in addition, for any $x, y\in\mathcal{P}$, $\xi\in V^{\ast}$, $v\in V$
\begin{eqnarray}
\psi_{\{,\}}(x\bullet y)=\psi_{\{,\}}(x)\psi_{\bullet}(y)
+(-1)^{{|x|}{|y|}}\psi_{\{,\}}(y)\psi_{\bullet}(x),
\end{eqnarray}
\begin{eqnarray}
\psi_{\bullet}(\{x, y\})=\psi_{\{,\}}(x)\psi_{\bullet}(y)
-(-1)^{{|x|}{|y|}}\psi_{\bullet}(y)\psi_{\{,\}}(x),
\end{eqnarray}
then $(V^{\ast},\psi^{\ast}_{\{,\}}, -\psi^{\ast}_{\bullet})$ is a representation of Poisson superalgebra $(\mathcal{P}, \{,\}, \bullet)$. Moreover, $(\mathcal{P}\ltimes V^{\ast}, \{,\}, \bullet)$ is also a Poisson superalgebra.
\end{thm}

\begin{proof} We know that if $(V,\psi_{\{,\}})$ is a representation of a Lie superalgebra $(\mathcal{P}, \{,\})$, then $(V^{\ast} ,\psi^{\ast}_{\{,\}})$ is a representation of the Lie superalgebra $(\mathcal{P}, \{,\})$. Also, if $(V, \psi_{\bullet})$ is a representation of a commutative associative superalgebra $(\mathcal{P}, \bullet)$, then $(V^{\ast}, -\psi^{\ast}_{\bullet})$ is a representation of the commutative associative superalgebra $(\mathcal{P}, \bullet)$. it remains to show
$$\psi^{\ast}_{\{,\}}(x\bullet y)+\psi^{\ast}_{\bullet}(y)\psi^{\ast}_{\{,\}}(x)
+(-1)^{{|x|}{|y|}}\psi^{\ast}_{\bullet}(x)\psi^{\ast}_{\{,\}}(y)=0,$$

$$-\psi^{\ast}_{\bullet}(\{x, y\})+\psi^{\ast}_{\{,\}}(x)\psi^{\ast}_{\bullet}(y)
-(-1)^{{|x|}{|y|}}\psi^{\ast}_{\bullet}(y)\psi^{\ast}_{\{,\}}(x)=0.$$
For this, we calculate
$$\langle  \textbf{(}\psi^{\ast}_{\{,\}}(x\bullet y)+\psi^{\ast}_{\bullet}(y)\psi^{\ast}_{\{,\}}(x)
+(-1)^{{|x|}{|y|}}\psi^{\ast}_{\bullet}(x)\psi^{\ast}_{\{,\}}(y)\textbf{)} \xi, v \rangle$$
\ \ \ \ \ \ \ \ \ \ \ \ \ \ \ \ \ \ \ $=\langle \xi, (-1)^{{{|x|}{|\xi|}}+{{|y|}{|\xi|}}} \textbf{(}-\psi_{\{,\}}(x\bullet y)+\psi_{\{,\}}(x)\psi_{\bullet}(y)
+(-1)^{{|x|}{|y|}}\psi_{\{,\}}(y)\psi_{\bullet}(x)\textbf{)} v\rangle=0,$

$$\langle  \textbf{(}-\psi^{\ast}_{\bullet}(\{x, y\})+\psi^{\ast}_{\{,\}}(x)\psi^{\ast}_{\bullet}(y)
-(-1)^{{|x|}{|y|}}\psi^{\ast}_{\bullet}(y)\psi^{\ast}_{\{,\}}(x)\textbf{)} \xi, v \rangle$$
\ \ \ \ \ \ \ \ \ \ \ \ \ \ \ \ \ \ \ $=\langle \xi, (-1)^{{{|x|}{|\xi|}}+{{|y|}{|\xi|}}} \textbf{(}\psi_{\bullet}(\{x, y\})-\psi_{\{,\}}(x)\psi_{\bullet}(y)
+(-1)^{{|x|}{|y|}}\psi_{\bullet}(y)\psi_{\{,\}}(x)\textbf{)} v\rangle=0.$
\\

So $(V^{\ast},\psi^{\ast}_{\{,\}}, -\psi^{\ast}_{\bullet})$ is a representation of Poisson superalgebra $(\mathcal{P}, \{,\}, \bullet)$. The remaining results follow from Proposition \ref{S()YYY} directly.
\end{proof}

\begin{example} Let $(\mathcal{P}, \{,\}, \bullet)$ be a Poisson superalgebra. Then $(\mathcal{P}, ad_{\{,\}}, L_{\bullet})$ and $(\mathcal{P}^{\ast}, ad_{\{,\}}^{\ast}, -L^{\ast}_{\bullet})$ are representation of Poisson superalgebra $(\mathcal{P}, \{,\}, \bullet)$.
\end{example}

In the sequel, we will present a method to construct a Poisson superalgebra structure on a direct sum $\mathcal{P}_1\oplus \mathcal{P}_2$ of the underlying vector spaces of two Poisson superalgebras $\mathcal{P}_1$
and $\mathcal{P}_2$ such that $\mathcal{P}_1$ and $\mathcal{P}_2$ are Poisson sub-superalgebra.

\begin{thm}\label{Theoreme} Let $(\mathcal{P}_1,\{,\}_1,\bullet_1)$ and $(\mathcal{P}_2,\{,\}_2,\bullet_2)$ be two Poisson superalgebras. Let $\rho_{\{,\}_1}, \varphi_{\bullet_1}: \mathcal{P}_1\longrightarrow End(\mathcal{P}_2)$ and $\rho_{\{,\}_2}, \varphi_{\bullet_2}: \mathcal{P}_2\longrightarrow End(\mathcal{P}_1)$ be four even linear maps such that $(\mathcal{P}_1, \mathcal{P}_2, \rho_{\{,\}_1}, \rho_{\{,\}_2})$ is a matched pair of Lie superalgebras and $(\mathcal{P}_1, \mathcal{P}_2, \varphi_{\bullet_1}, \varphi_{\bullet_2})$ is a matched pair of commutative associative superalgebras. If in addition $(\mathcal{P}_2, \rho_{\{,\}_1}, \varphi_{\bullet_1})$ and $(\mathcal{P}_1, \rho_{\{,\}_2}, \varphi_{\bullet_2})$ are representations of the Poisson superalgebras $(\mathcal{P}_1,\{,\}_1,\bullet_1)$ and $(\mathcal{P}_2,\{,\}_2,\bullet_2)$ respectively, and $\rho_{\{,\}_1}, \rho_{\{,\}_2}, \varphi_{\bullet_1}, \varphi_{\bullet_2}$ are compatible in the following
sense:
\begin{align}\label{Hom 1}
\rho_{\{,\}_2}(a)(x\bullet_1 y )&=(\rho_{\{,\}_2}(a)x)\bullet_1 y +(-1)^{{|a|}{|x|}} x \bullet_1 (\rho_{\{,\}_2}(a)y)
\\&-(-1)^{{|a|}{|x|}}\varphi_{\bullet_2}(\rho_{\{,\}_{1}}(x)a)y
-(-1)^{{{|a|}{|y|}}+{{|x|}{|y|}}}\varphi_{\bullet_2}(\rho_{\{,\}_{1}}(y)a)x,
\nonumber\end{align}
\begin{align}\label{Hom 2}
\{x,\varphi_{\bullet_2}(a)(y)\}_1-(-1)^{{{|x|}{|a|}}+{{|x|}{|y|}}+{{|a|}{|y|}}}\rho_{\{,\}_2}(\varphi_{\bullet_1}(y)a)x &=\varphi_{\bullet_2}(\rho_{\{,\}_{1}}(x)a)y-(-1)^{{|x|}{|a|}}(\rho_{\{,\}_2}(a)x)\bullet_1y\\
&+(-1)^{{|x|}{|a|}}\varphi_{\bullet_2}(a)(\{x,y\}_1),
\nonumber\end{align}
\begin{align}\label{Hom 3}
\rho_{\{,\}_1}(x)(a\bullet_2 b )&=(\rho_{\{,\}_1}(x)a)\bullet_2b +(-1)^{{|x|}{|a|}}a\bullet_2 (\rho_{\{,\}_1}(x)b)\\&-(-1)^{{|x|}{|a|}}\varphi_{\bullet_1}(\rho_{\{,\}_{2}}(a)x)b
-(-1)^{{{|x|}{|b|}}+{{|a|}{|b|}}}\varphi_{\bullet_1}(\rho_{\{,\}_{2}}(b)x)a,
\nonumber\end{align}
\begin{align}\label{Hom 4}
\{a,\varphi_{\bullet_1}(x)(b)\}_2-(-1)^{{{|a|}{|x|}}+{{|a|}{|b|}}+{{|x|}{|b|}}}\rho_{\{,\}_1}(\varphi_{\bullet_2}(b)x)a &=\varphi_{\bullet_1}(\rho_{\{,\}_{2}}(a)x)b-(-1)^{{|a|}{|x|}}(\rho_{\{,\}_1}(x)a)\bullet_2 b\\&+(-1)^{{|a|}{|x|}}\varphi_{\bullet_1}(x)(\{a,b\}_2),
\nonumber\end{align}
for any $x,y\in \mathcal{P}_1$ and $a,b\in \mathcal{P}_2$,\\ then there exists a Poisson superalgebra structure $\{,\}$, $\bullet$ on the vector superspace $\mathcal{P}_1\oplus \mathcal{P}_2$ given by
$\{x+a,\ y+b\}= (\{x, y\}_{1}+\rho_{\{,\}_2}(a)y-(-1)^{{|x|}{|y|}}\rho_{\{,\}_2}(b)x)+(\{a, b\}_{2}+\rho_{\{,\}_1}(x)b-(-1)^{{|x|}{|y|}}\rho_{\{,\}_1}(y)a),$

$(x+a)\bullet (y+b)= (x\bullet_1 y+\varphi_{\bullet_2}(a)y+(-1)^{{|x|}{|y|}}\varphi_{\bullet_2}(b)x)+ (a\bullet_2 b+\varphi_{\bullet_1}(x)b+(-1)^{{|x|}{|y|}}\varphi_{\bullet_1}(y)a).$
\end{thm}

We denote this Poisson superalgebra by $\mathcal{P}_1 \bowtie^{\rho_{\{,\}_2}, \varphi_{\bullet_2}}_{\rho_{\{,\}_1}, \varphi_{\bullet_1}}\mathcal{P}_2$ or simply $\mathcal{P}_1\bowtie\mathcal{P}_2$. 

Moreover, $(\mathcal{P}_1, \mathcal{P}_2, \rho_{\{,\}_1}, \varphi_{\bullet_1}, \rho_{\{,\}_2}, \varphi_{\bullet_2})$ satisfying the above conditions is called a matched pair of Poisson superalgebras. On the other hand, every Poisson superalgebra which is a direct sum of the underlying superspaces of sub-superalgebra
can be obtained in the above way.

\begin{proof} If fact, $\mathcal{P}_1\oplus \mathcal{P}_2$ becomes a Poisson superalgebra if and only if the following equations are satisfied (for any $x, y\in\mathcal{P}_1$, $a, b\in\mathcal{P}_2$):
\begin{eqnarray}\label{super nouveau 1}
(-1)^{{|a|}{|y|}}\{a,\{x,y\}\}+(-1)^{{|x|}{|a|}}\{x,\{y,a\}\}+(-1)^{{|y|}{|x|}}\{y,\{a,x\}\}=0, \end{eqnarray}
\begin{eqnarray}\label{super nouveau 2}
(-1)^{{|x|}{|b|}}\{x,\{a,b\}\}+(-1)^{{|a|}{|x|}}\{a,\{b,x\}\}+(-1)^{{|b|}{|a|}}\{b,\{x,a\}\}=0, \end{eqnarray}
\begin{eqnarray}\label{ETE}
(a\bullet x)\bullet y=a\bullet( x\bullet y), \ (x\bullet a)\bullet y=x\bullet( a\bullet y), \ (x\bullet a)\bullet b=x\bullet( a\bullet b), \ (a\bullet x)\bullet b=a\bullet( x\bullet b),
\end{eqnarray}
\begin{eqnarray}\label{super POI 1}
\{a,x\bullet y\}=\{a,x\}\bullet y+(-1)^{{|a|}{|x|}}x\bullet\{a,y\}, \ \{x,a\bullet y\}=\{x,a\}\bullet y+(-1)^{{|a|}{|x|}}a\bullet\{x,y\},
\end{eqnarray}
\begin{eqnarray}\label{super POI 2}
\{x,a\bullet b\}=\{x,a\}\bullet b+(-1)^{{|a|}{|x|}}a\bullet\{x,b\}, \ \{a,x\bullet b\}=\{a,x\}\bullet b+(-1)^{{|a|}{|x|}}x\bullet\{a,b\}.
\end{eqnarray}

Note that Eq. (\ref{super nouveau 1}) and (\ref{super nouveau 2})  is equivalent to the fact that $(\mathcal{P}_1, \mathcal{P}_2, \rho_{\{,\}_1}, \rho_{\{,\}_2})$ is a matched pair of Lie superalgebras and Eq. (\ref{ETE}) are equivalent to the fact that $(\mathcal{P}_1, \mathcal{P}_2, \varphi_{\bullet_1}, \varphi_{\bullet_2})$ is a matched pair of commutative associative superalgebras. Moreover, Eqs.(\ref{super POI 1}) and (\ref{super POI 1}) are equivalent to the fact that $(\rho_{\{,\}_1}, \varphi_{\bullet_1})$ and $(\rho_{\{,\}_2}, \varphi_{\bullet_2})$ satisfy the compatibility conditions (\ref{super rep 1}) and (\ref{super rep 2}) and Eqs. (\ref{Hom 1}), (\ref{Hom 2}), (\ref{Hom 3}), (\ref{Hom 4}).
\end{proof}

\section{Poisson superbialgebras}
We recall first  some basic notions and facts about Lie superbialgebras and commutative and cocommutative infinitesimal superbialgebras.

\begin{df} A Lie supercoalgebra is a pair $(\mathcal{A}, \delta)$ consisting of a superspace $\mathcal{A}$ and a linear map $\delta:\mathcal{A} \longrightarrow \mathcal{A}\otimes\mathcal{A}$ satisfying the following conditions:
\begin{eqnarray}\label{cocrochet}
\delta(\mathcal{A}^i) \subset \sum_{i=j+k} \mathcal{A}^j\otimes\mathcal{A}^k \ \  for \ \  i\in\mathbb{Z}_{2},
\end{eqnarray}
\begin{eqnarray}
 Im\delta\subset Im(id\otimes id-\tau)  \ \  i.e. \ \ \delta \ is \ skew-supersymmetric,
\end{eqnarray}
\begin{eqnarray}
(id\otimes id\otimes id+\xi+\xi^2)\circ(id\otimes\delta)\circ\delta=0 : \mathcal{A}\rightarrow \mathcal{A}\otimes\mathcal{A}\otimes\mathcal{A}.
\end{eqnarray}
\end{df}

\begin{df} A coassociative supercoalgebra is a pair $(\mathcal{A}, \Delta)$ consisting of a superspace $\mathcal{A}$ and a  linear map $\Delta:\mathcal{A} \longrightarrow \mathcal{A}\otimes\mathcal{A}$ satisfying the following conditions:
\begin{eqnarray}
\Delta(\mathcal{A}^i) \subset \sum_{i=j+k} \mathcal{A}^j\otimes\mathcal{A}^k \ \  for \ \  i\in\mathbb{Z}_{2},
\end{eqnarray}
\begin{eqnarray}
(id\otimes\Delta)\circ\Delta=(\Delta\otimes id)\circ\Delta  \ \ \ \ \ (coassociativity).
\end{eqnarray}
\end{df}

\begin{rem}
\textbf{1)} If $(\mathcal{A}, [\cdot ,\cdot ])$ is a Lie superalgebra, we let $ad_{[\cdot ,\cdot ]}(x)=ad(x)$ denote the adjoint operator, that is, $ad_{[\cdot ,\cdot ]}(x)y=ad(x)y=[x, y]$ for any $x, y\in\mathcal{A}$. For an element $x$ in a Lie superalgebra $(\mathcal{A}, [\cdot ,\cdot ])$ and $n\geq 2$, define the adjoint map $ad(x):\mathcal{A}^{\otimes n}\rightarrow \mathcal{A}^{\otimes n}$ by
$$ad(x)(y_{1}\otimes\cdot\cdot\cdot\otimes y_{n})=\sum_{i=1}^{n} (-1)^{{|x|}({|y_{1}|+{|y_{2}|}}+\cdot\cdot\cdot+{|y_{i-1}|})}y_{1}\otimes\cdot\cdot\cdot\otimes y_{i-1}\otimes[x, y_{i}]\otimes y_{i+1}\cdot\cdot\cdot\otimes y_{n}.$$
For  $n=2$, $ad(x)(y_{1}\otimes y_{2})=[x, y_{1}]\otimes y_{2}+(-1)^{{|x|}{|y_{1}|}}y_{1}\otimes[x, y_{2}]$. Conversely, given $\gamma=y_{1}\otimes\cdot\cdot\cdot\otimes y_{n}$, we define the map $ad(\gamma):\mathcal{A}\rightarrow \mathcal{A}^{\otimes n}$ by $ad(\gamma)(x)=ad(x)(\gamma)$, for $x\in\mathcal{A}$.

\textbf{2)} Let $(\mathcal{A}, \bullet)$ be an associative superalgebra with a bilinear map $\bullet:\mathcal{A}\times\mathcal{A}\longrightarrow \mathcal{A}$, $L_{\bullet}(x)$ and let $R_{\bullet}(x)$ denote the left and right multiplication operator respectively, that is, $L_{\bullet}(x)y=(-1)^{{|x|}{|y|}}R_{\bullet}(y)x=x\bullet y$, for all homogeneous elements  $x, y \in\mathcal{A}$.

We also simply denote them by $L(x)$ and $R(x)$ respectively when there is no confusion.
\end{rem}

\begin{df}\label{ae} \cite{Hengyun Y. and Yucai S} A Lie superbialgebra is a triple $(\mathcal{A}, [\cdot,\cdot], \delta)$ such that
\begin{enumerate}
\item $(\mathcal{A}, [\cdot,\cdot])$ is a Lie superalgebra.
\item $(\mathcal{A}, \delta)$ is a Lie supercoalgebra.
\item The following compatibility condition holds for all $x, y \in\mathcal{A}$ :
\begin{eqnarray}\label{a}
\delta([x,y])=(ad(x)\otimes id+id\otimes ad(x))\delta(y)-(-1)^{{|x|}{|y|}}(ad(y)\otimes id+id\otimes ad(y))\delta(x).
\end{eqnarray}
\end{enumerate}
\end{df}

 The map $f:\mathcal{A} \rightarrow \mathcal{A}'$ is called  even (resp. odd) map if $f(\mathcal{A}_{i})\subset \mathcal{A}'_{i}$ (resp. $f(\mathcal{A}_{i})\subset\mathcal{A}'_{i+1}),$ for $i=0, 1$.
A morphism of Lie superbialgebras is an even  linear map such that \\
$$f\circ[\cdot ,\cdot ]=[\cdot ,\cdot ]\circ f^{\otimes2}\ \ \ \      \text{and }\ \ \  \delta\circ f=f^{\otimes2}\circ\delta.$$\\
An isomorphism of Lie superbialgebras is an invertible morphism of Lie superbialgebras. Two Lie superbialgebras are said to be isomorphic if there exists an isomorphism between them.

\begin{df}\label{aeee} An infinitesimal superbialgebra is a triple $(\mathcal{A}, \bullet, \Delta)$ such that
\begin{enumerate}
\item $(\mathcal{A}, \bullet)$ is a associative superalgebra,
\item $(\mathcal{A}, \Delta)$ is a coassociative
supercoalgebra,
\item The following compatibility condition holds for all $a, b \in\mathcal{A}$ :
\begin{eqnarray}\label{aee}
\Delta(a\bullet b)=(L_{\bullet}(a)\otimes id)\Delta(b)+(id\otimes R_{\bullet}(b))\Delta(a).
\end{eqnarray}
\end{enumerate}
\end{df}
A morphism of infinitesimal superbialgebras is a linear map that commutes
the multiplications and the comultiplications.

\begin{df} Let $(\mathcal{A}, \{,\}, \bullet)$ be a Poisson superalgebra. Suppose that it is equipped with two comultiplications $\delta, \ \Delta: \mathcal{A}\rightarrow \mathcal{A}\otimes\mathcal{A}$ such that $(\mathcal{A}, \delta, \Delta)$ is a Poisson supercoalgebra, that is, $(\mathcal{A}, \delta)$ is a Lie supercoalgebra and  $(\mathcal{A}, \Delta)$ is a coassociative supercoalgebra which is cocommutative and they satisfy the
following compatible condition for all $x\in\mathcal{A}$:
\begin{eqnarray}\label{aa}
(id\otimes\Delta)\delta(x)=(\delta\otimes id)\Delta(x)+(\tau\otimes id)(id\otimes\delta)\Delta(x).
\end{eqnarray}
If in addition, $(\mathcal{A}, \{,\}, \delta)$ is a Lie superbialgebra and $(\mathcal{A}, \bullet, \Delta)$ is a commutative and cocommutative infinitesimal superbialgebra and $\delta$, $\Delta$ are compatible in the following sense, for all $x, y\in\mathcal{A}$, $a, b\in\mathcal{A}^{\ast}$:
\begin{align}\label{atttt}
\langle \delta(x\bullet y), b\otimes a\rangle &=\langle (-1)^{{|x|}{|y|}}(L_{\bullet}(y)\otimes id)\delta(x)+(L_{\bullet}(x)\otimes id)\delta(y)\\
&+(-1)^{{|x|}{|b|}}(id\otimes ad_{\{, \}}(x))\Delta(y)+(-1)^{{{|y|}{|b|}}+{{|x|}{|y|}}}(id\otimes ad_{\{, \}}(y))\Delta(x), b\otimes a\rangle,
\nonumber\end{align}

\begin{align}\label{aaattt}
\langle \Delta(\{x, y\}), b\otimes a\rangle &=\langle(ad_{\{, \}}(x)\otimes id+(-1)^{{|x|}{|b|}}id\otimes ad_{\{, \}}(x))\Delta(y)\\
&+(-1)^{{|x|}{|y|}}(L_{\bullet}(y)\otimes id-(-1)^{{|y|}{|b|}}id\otimes L_{\bullet}(y))\delta(x), b\otimes a\rangle,
\nonumber\end{align}

then $(\mathcal{A}, \{,\}, \bullet, \delta, \Delta)$ is called  Poisson superbialgebra.
\end{df}

A homomorphism between two Poisson supercoalgebras is defined as a $\mathbb{K}$-linear map between the two Poisson supercoalgebras that preserves the corresponding cooperations. A homomorphism between two Poisson superbialgebras is a homomorphism of both Poisson superalgebras and Poisson supercoalgebras.

\begin{example} Let $\mathcal{A}=\mathcal{A}_{\bar{0}}\oplus\mathcal{A}_{\bar{1}}$ be a 2-dimensional superspace where $\mathcal{A}_{\bar{0}}$ is generated by $e_{1}$ and $\mathcal{A}_{\bar{1}}$ is generated by $e_{2}$. Then  $(\mathcal{A}, \{,\}, \bullet, \delta, \Delta)$ is a Poisson superbialgebra in the following situation :\\
The pair $(\mathcal{A}, \{,\})$ is a Lie superalgebra when $$\{e_{1}, e_{2}\}=b e_{2}, \ \ \{e_{2}, e_{2}\}=c e_{1}, \ \ \{e_{1}, e_{1}\}=0, \ \ with \ \ \ \ b c=0.$$
The multiplication $\bullet$ and $\Delta$ are  of the form :\\
$$e_{1}\bullet e_{2}=d e_{2}, \ \  e_{1}\bullet e_{1}=k e_{1}, \ \  e_{2}\bullet e_{2}=f e_{1} \ \ and \ \ \ \Delta(e_{1})=\Delta(e_{2})=0.$$
The pair $(\mathcal{A}, \bullet)$ is a commutative associative superalgebra when $ d=k$.\\
The cobracket $\delta$ is of the form :
$\delta(e_{1})=c_{1}e_{2}\otimes e_{2}, \  \delta(e_{2})=c_{2}(e_{1}\otimes e_{2}-e_{2}\otimes e_{1})$.\\
The triple $(\mathcal{A}, \{,\}, \delta)$ is a Lie superbialgebra if \ $c_{1}c_{2}=0$ and $cc_{1}=-4 b c_{2}$.\\
Also  $\delta$ and $\Delta$ are compatible if
$k c_{1}=2 d c_{1}$ and $fc_{2}=0$, where $b, c, d, k, f, c_{1}, c_{2}$ are parameters in $\mathbb{K}$.
Therefore,   we have a Poisson superalgebra for $c=d=k=c_2=0$ or $b=f=c_1=0$ and $d=k$. 
\end{example}

\section{ Manin triples of Poisson superalgebras and Poisson superbialgebras}

We introduce first a notion of Manin triple of Poisson superalgebras which is an analogue of the notion of Manin triple for Lie superalgebras \cite{newapproach-Bai, V. Chari}.

\begin{df} A Manin triple of Poisson superalgebras $(\mathcal{P}, \mathcal{P}^{+}, \mathcal{P}^{-})$ is a triple of Poisson superalgebras $\mathcal{P}$, $\mathcal{P}^{+}$, and $\mathcal{P}^{-}$ together with a nondegenerate supersymmetric bilinear form $B$ on $\mathcal{P}$ which is invariant in the sense that
\begin{eqnarray}
B(\{x, y\}, z)=B(x, \{y, z\}), \ \ \ \ \ \ \ \ \ \ \ \ B(x\bullet y, z)=B(x, y\bullet z),
\end{eqnarray}
for  any $x, y, z\in \mathcal{P}$, satisfying the following conditions:
\begin{enumerate}
\item  $\mathcal{P}^{+}$ and $\mathcal{P}^{-}$ are Poisson sub-superalgebra of $\mathcal{P}$,
\item  $\mathcal{P}=\mathcal{P}^{+}\oplus \mathcal{P}^{-}$ as $\mathbb{K}$-vector superspace,
\item  $\mathcal{P}^{+}$ and $\mathcal{P}^{-}$ are isotropic with respect to $B$.

\end{enumerate}
\end{df}
\
\\
A homomorphism between two Manin triples of Poisson superalgebras $(\mathcal{P}_{1}, \mathcal{P}_{1}^{+}, \mathcal{P}_{1}^{-})$ and $(\mathcal{P}_{2}, \mathcal{P}_{2}^{+}, \mathcal{P}_{2}^{-})$ with two nondegenerate supersymmetric invariant bilinear forms $B_{1}$ and $B_{2}$ respectively is a homomorphism of Poisson superalgebras $\varphi: \mathcal{P}_{1}\rightarrow \mathcal{P}_{2}$ such that
\begin{eqnarray}
\varphi(\mathcal{P}_{1}^{+})\subset \mathcal{P}_{2} ^{+}, \ \ \ \  \varphi(\mathcal{P}_{1}^{-})\subset \mathcal{P}_{2} ^{-}, \ \ \  B_{1}(x, y)=\varphi^{\ast}B_{2}(x, y)=B_{1}(\varphi(x), \varphi(y)),
\end{eqnarray}
for  any $x, y\in \mathcal{P}_{1}$.\\

Obviously, a Manin triple of Poisson superalgebras is just a triple of Poisson superalgebras such that they are both a Manin triple of Lie superalgebra and a commutative associative version of Manin triple with the same nondegenerate supersymmetric bilinear form (and share the same isotropic sub-superalgebra). Moreover, it is easy to see that $\mathcal{P}^{+}$ and $\mathcal{P}^{-}$ are both Lagrangian sub-superalgebras of $\mathcal{P}$.\\
In particular, there is a special (standard) Manin triple of Poisson superbialgebras as follows. Let $(\mathcal{P}, \{,\}, \bullet)$ be a Poisson superalgebra. If there is a Poisson superalgebra structure on the direct sum of the
underlying vector superspace of $P$ and its dual superspace $\mathcal{P}^{\ast}$ such that $(\mathcal{P}\oplus \mathcal{P}^{\ast}, \mathcal{P}, \mathcal{P}^{\ast})$ is a Manin triple of Poisson superalgebras with the invariant supersymmetric bilinear form on $\mathcal{P}\oplus \mathcal{P}^{\ast}$ given by
\begin{eqnarray}\label{Manin}
B_{p}(x+a, y+b)=\langle x, b  \rangle + \langle a, y \rangle, \ \ \ \ \ \ \  for \ any \ x, y \in \mathcal{P}, \ \ a, b\in \mathcal{P}^{\ast},
\end{eqnarray}
then $(\mathcal{P}\oplus \mathcal{P}^{\ast}, \mathcal{P}, \mathcal{P}^{\ast})$ is called a standard Manin triple of Poisson superalgebras.\\
Obviously, a standard Manin triple of Poisson superalgebras is a Manin triple of Poisson superalgebras. Conversely, it is easy to show that every Manin triple of
Poisson superalgebras is isomorphic to a standard
one. Furthermore, it is straightforward to get the following structure Theorem (\cite{V. Chari} and \cite{C. Bai 1}).\\

In the following, we concentrate on the case that $\mathcal{P}_1$ is $\mathcal{P}$ and $\mathcal{P}_2$ is $\mathcal{P^{\ast}}$, where $\mathcal{P^{\ast}}$ is the dual space of $\mathcal{P}$, and $\rho_{\{,\}_i}=ad^{\ast}_{\{,\}_i}$, $\varphi_{\bullet_i}=-L^{\ast}_{\bullet_i}$ $(i=1, 2)$.

\begin{thm} Let $(\mathcal{P} ,\{,\}_1,\bullet_1)$ and $(\mathcal{P^{\ast}}, \{,\}_2,\bullet_2)$ be two Poisson superalgebras. Then $(\mathcal{P}\oplus \mathcal{P}^{\ast}, \mathcal{P}, \mathcal{P}^{\ast})$ is a standard Manin triple of Poisson superalgebras if and only if $(\mathcal{P}, \mathcal{P}^{\ast}, ad^{\ast}_{\{,\}_1}, -L^{\ast}_{\bullet_1}, ad^{\ast}_{\{,\}_2}, -L^{\ast}_{\bullet_2})$ is a matched pair of Poisson superalgebras.
\end{thm}
Like the Manin triples for Lie superalgebra corresponding to Lie superbialgebras and the commutative associative versions of Manin triples corresponding to commutative and cocommutative infinitesimal superbialgebras, \cite{M. Aguiar},  \cite{V. N. Zhelyabin} there is also a bialgebra structure which corresponds to a Manin triple of Poisson superalgebras following from the equivalent conditions of matched pairs of Poisson superalgebras.\\

\begin{thm} Let $(\mathcal{P} ,\{,\}_1,\bullet_1)$ and  $(\mathcal{P^{\ast}}, \{,\}_2,\bullet_2)$ be two  Poisson superalgebras equipped with two comultiplications $\delta, \ \Delta: \mathcal{P}\rightarrow \mathcal{P}\otimes\mathcal{P}$. Suppose that $\delta^{\ast}, \ \Delta^{\ast}: \mathcal{P}^{\ast}\otimes\mathcal{P}^{\ast}\subset(\mathcal{P}\otimes\mathcal{P})^{\ast}\rightarrow \mathcal{P}^{\ast}$ induce a Poisson superalgebra structure on $\mathcal{P}^{\ast}$, where $\delta^{\ast}$ and $\Delta^{\ast}$ correspond to the Lie bracket and the product of the commutative associative superalgebra respectively. Set $\{,\}_i=\delta^{\ast}$, $\bullet_i=\Delta^{\ast}$ $(i=1, 2)$. Then the following conditions are equivalent:
\begin{enumerate}
\item $(\mathcal{P}, \{,\}_1,\bullet_1, \delta, \Delta)$ and $(\mathcal{P^{\ast}}, \{,\}_2,\bullet_2, \delta, \Delta)$  are Poisson superbialgebras.
\item $(\mathcal{P}, \mathcal{P}^{\ast}, ad^{\ast}_{\{,\}_1}, -L^{\ast}_{\bullet_1}, ad^{\ast}_{\{,\}_2}, -L^{\ast}_{\bullet_2})$ is a matched pair of Poisson superalgebras.
\item $(\mathcal{P}\oplus \mathcal{P}^{\ast}, \mathcal{P}, \mathcal{P}^{\ast})$ is a standard Manin triple of Poisson superalgebras with the bilinear form defined by Eq. (\ref{Manin}) and the isotropic subsuperalgebras are $\mathcal{P}$ and $\mathcal{P}^{\ast}$.
\end{enumerate}
\end{thm}

\begin{proof} We only need to prove that the fact $(1)$ holds if and only if the fact $(2)$ holds. In fact, it is known that $(\mathcal{P}, \{,\}, \delta)$ is a Lie superbialgebra if and only if $(\mathcal{P}, \mathcal{P}^{\ast}, ad^{\ast}_{\{,\}_1}, ad^{\ast}_{\{,\}_2})$ is a matched
pair of Lie superalgebra \cite{newapproach-Bai} and $(\mathcal{P}, \bullet, \Delta)$ is a commutative and cocommutative infinitesimal superbialgebra if and only if $(\mathcal{P}, \mathcal{P}^{\ast}, -L^{\ast}_{\bullet_1}, -L^{\ast}_{\bullet_2})$ is a matched pair of commutative associative superalgebras. Then by Theorem \ref{Theoreme}, we only need to prove Eqs. (\ref{Hom 1})$-$(\ref{Hom 4}) are equivalent to Eqs. (\ref{atttt}) and (\ref{aaattt}) in the case that $\rho_{\{,\}_i}=ad^{\ast}_{\{,\}_i}$, $\varphi_{\bullet_i}=-L^{\ast}_{\bullet_i}$ $(i=1, 2)$. As an example, we give an explicit proof of the case that (for any $x, y\in\mathcal{P}$, $a\in\mathcal{P}^{\ast}$)
\begin{align*}-ad^{\ast}_{\{,\}_2}(a)(x\bullet_1y)+(ad^{\ast}_{\{,\}_2}(a)x)\bullet_1y
+(-1)^{{|a|}{|x|}}x\bullet_1(ad^{\ast}_{\{,\}_2}(a)y)\\
+(-1)^{{|a|}{|x|}}L^{\ast}_{\bullet_2}(ad^{\ast}_{\{,\}_1}(x)a)y
+(-1)^{{{|a|}{|y|}}+{{|x|}{|y|}}}L^{\ast}_{\bullet_2}(ad^{\ast}_{\{,\}_1}(y)a)x=0\end{align*}
is equivalent to Eq. (\ref{atttt}). The proof of other cases is similar. In fact, let the left hand side of the above equation acts on an arbitrary element $b\in\mathcal{P}^{\ast}$, we have
\begin{align*}
0&=\langle-ad^{\ast}_{\{,\}_2}(a)(x\bullet_1y)+(ad^{\ast}_{\{,\}_2}(a)x)\bullet_1y
+(-1)^{{|a|}{|x|}}x\bullet_1(ad^{\ast}_{\{,\}_2}(a)y)
+(-1)^{{|a|}{|x|}}L^{\ast}_{\bullet_2}(ad^{\ast}_{\{,\}_1}(x)a)y\\
&+(-1)^{{{|a|}{|y|}}+{{|x|}{|y|}}}L^{\ast}_{\bullet_2}(ad^{\ast}_{\{,\}_1}(y)a)x, b\rangle\\
&=-(-1)^{{{|a|}{|b|}}+{{|a|}{|x|}}+{{|a|}{|y|}}}\langle x\bullet_1y, {\{b, a\}_2}\rangle+(-1)^{{|a|}{|x|}}\langle x, {\{a, L^{\ast}_{\bullet_1}(y)b\}_2}\rangle+(-1)^{{{|x|}{|y|}}+{{|a|}{|y|}}} \langle y, {\{a, L^{\ast}_{\bullet_1}(x)b\}_2}\rangle\\
&-(-1)^{{{|a|}{|b|}}+{{|a|}{|x|}}+{{|a|}{|y|}}+{{|x|}{|b|}}+{{|x|}{|y|}}}\langle y, b\bullet_2 ad^{\ast}_{\{,\}_1}(x)a\rangle-(-1)^{{{|a|}{|y|}}+{{|a|}{|x|}}}\langle x, ad^{\ast}_{\{,\}_1}(y)a\bullet_2 b\rangle\\
&=(-1)^{{{|a|}{|b|}}+{{|a|}{|x|}}+{{|a|}{|y|}}}\langle -\delta(x\bullet_1y)+(-1)^{{{|x|}{|y|}}}(L_{\bullet_1}(y)\otimes id)\delta(x)
+(L_{\bullet_1}(x)\otimes id)\delta(y)\\ & +(-1)^{{{|x|}{|b|}}}(id\otimes ad_{\{,\}_1}(x))\Delta(y)
+(-1)^{{{|y|}{|b|}}+{{|x|}{|y|}}}(id\otimes ad_{\{,\}_1}(y))\Delta(x), b\otimes a\rangle.
\end{align*}
Then the conclusion follows.
\end{proof}

Furthermore, it is easy to show that two Manin triples of Poisson superalgebras are isomorphic if and only if their corresponding Poisson superbialgebras are isomorphic.

\section{Coboundary Poisson superbialgebras}

We recall first some relevant results of  coboundary Lie superbialgebras and  coboundary infinitesimal superbialgebras. For example, they lead to the famous classical Yang-Baxter equation (CYBE) \cite{V. Chari}, \cite{V. G. Drinfeldd} and associative Yang-Baxter equation (AYBE) \cite{M. Aguiarrr}, respectively.

\begin{df}
Denote by  ${U}(\mathcal{A})$, the universal enveloping algebra of a superspace $\mathcal{A}$. If $r=\sum r_{1}\otimes r_{2}$ $\in\mathcal{A}\otimes\mathcal{A}$, then
(here we also use $id$ to denote the identity element of ${U}(\mathcal{A})$)
\begin{eqnarray}\label{10}
  &&
  r_{12}=\sum r_{1}\otimes r_{2}\otimes id=r\otimes id,
   \\ \nonumber
  &&
 r_{13}=\sum r_{1}\otimes id \otimes r_{2}=(id \otimes\tau)(r\otimes id)=(\tau\otimes id)(id\otimes r),\\ \nonumber
  &&
  r_{23}=\sum id\otimes r_{1}\otimes r_{2}=id\otimes r
 \end{eqnarray}



are elements of ${U}(\mathcal{A})\otimes{U}(\mathcal{A})\otimes{U}(\mathcal{A})$,
\end{df}

\begin{df}
 The classical Yang-Baxter equation (CYBE) in a Lie superalgebra $(\mathcal{A}, [\cdot,\cdot])$ is
\begin{eqnarray}\label{1000}
C(r)=[r_{12}, r_{13}]+[r_{12}, r_{23}]+[r_{13}, r_{23}]=0,
\end{eqnarray}
for $r\in\mathcal{A}\otimes\mathcal{A}$. The three brackets in (\ref{1000}) are defined as \\
\begin{eqnarray}\label{cvc}
  &&
   [r_{12},r'_{13}]=\sum (-1)^{{| r'_{1}|}{|r_{2}|}} [r_{1}, r'_{1}]\otimes r_{2}\otimes r'_{2},
   \\ \nonumber
  &&
 [r_{12},r'_{23}]=\sum r_{1}\otimes[r_{2},r'_{1}]\otimes r'_{2},   \\ \nonumber
  &&
  [r_{13},r'_{23}]=\sum  (-1)^{{|r'_{1}|}{| r_{2}|}} r_{1} \otimes r'_{1} \otimes [r_{2}, r'_{2}],
 \end{eqnarray}
 where $r=\sum r_{1}\otimes r_{2}$ and $r'=\sum r'_{1}\otimes r'_{2}$ $\in\mathcal{A}\otimes\mathcal{A}$.
\end{df}

\begin{df} The associative Yang-Baxter equation (AYBE) $($it is also called AYB in Ref. \cite{M. Aguiar 2}$)$ in a associative superalgebra $(\mathcal{A}, \bullet)$ is
\begin{eqnarray}\label{Baxter Hom-associative}
A(r)=r_{13}\bullet r_{12}-r_{12}\bullet r_{23}+r_{23}\bullet r_{13}=0, \ \ for \ r\in\mathcal{A}\otimes\mathcal{A}.
\end{eqnarray}
The three elements in (\ref{Baxter Hom-associative}) are defined as \\
\begin{eqnarray}\label{cvc 2}
  &&
  r_{13}\bullet r'_{12}=\sum r_{1}\bullet r'_{1}\otimes r'_{2}\otimes r_{2},
   \\ \nonumber
  &&
  r_{12}\bullet r'_{23}=\sum r_{1}\otimes r_{2}\bullet r'_{1}\otimes r'_{2},   \\ \nonumber
  &&
  r_{23}\bullet r'_{13}=\sum  r'_{1} \otimes r_{1} \otimes r_{2}\bullet r'_{2},
 \end{eqnarray}
 where $r=\sum r_{1}\otimes r_{2}$ and $r'=\sum r'_{1}\otimes r'_{2}$ $\in\mathcal{A}\otimes\mathcal{A}$.

\end{df}

\begin{df} A coboundary Lie superbialgebra $(\mathcal{A}, [\cdot,\cdot], \delta, r)$ consists of a Lie superbialgebra $(\mathcal{A}, [\cdot,\cdot], \delta)$ and an element $r=\sum r_{1}\otimes r_{2}\in\mathcal{A}\otimes\mathcal{A}$
such that for all $x \in \mathcal{A}$
\begin{eqnarray}\label{cobord 1}
\ \ \delta(x)=(ad(x)\otimes id +(-1)^{{|x|}{|r_{1}|}}id\otimes ad(x))r=[x, r_{1}]\otimes r_{2}+ (-1)^{{|x|}{|r_{1}|}}r_{1}\otimes[x, r_{2}],
\end{eqnarray}
\end{df}

\begin{df} A coboundary infinitesimal superbialgebra  $(\mathcal{A}, \bullet, \Delta, r)$ consists of an infinitesimal superbialgebra $(\mathcal{A}, \bullet, \Delta)$ and an element $r=\sum r_{1}\otimes r_{2}\in\mathcal{A}\otimes\mathcal{A}$
such that $|r_{1}|=|r_{2}|$ and  for all $ a \in \mathcal{A}$
\begin{eqnarray}\label{cobord 2}
\Delta(a)=(L_{\bullet}(a)\otimes id-(-1)^{{|a|}{|r_{1}|}}id \otimes R_{\bullet}(a))r=a\bullet r_{1}\otimes  r_{2}-r_{1}\otimes r_{2}\bullet a.
\end{eqnarray}
\end{df}

\begin{rem}
\begin{equation}\label{super cobordd}
\delta(x)=(-1)^{{|x|}{|r|}}(\sum [x, r_{1}]\otimes r_{2}+(-1)^{{|x|}{| r_{1}|}} r_{1}\otimes[x, r_{2}])
\end{equation}
for $x\in\mathcal{A}$, where the parity $|r|$ of $r$ is defined as follows : since we assume $r$ is homogenous, there exists $|r|\in\mathbb{Z}_{2}$, such that $r$ can be written as $r=\sum r_{1}\otimes r_{2}\in\mathcal{A}^{\otimes2}$, $r_{1}, r_{2}$ are homogenous elements with $|r|=|r_{1}|+|r_{2}|$. (Note that equations (\ref{super cobordd}) and (\ref{cocrochet}) show that we  have $|r|=\bar{0}$,  namely $|r_{1}|=|r_{2}|$). So we get (\ref{cobord 1}).
\end{rem}

Let $(\mathcal{A}, [\cdot,\cdot])$ be a Lie superalgebra and $r\in\mathcal{A}\otimes\mathcal{A}$. The linear map $\delta$ defined by Eq. (\ref{cobord 1}) and $|r|=\bar{0}$ makes $(\mathcal{A}, \delta)$ into a Lie supercoalgebra if and only if the following conditions are satisfied (for any $x\in\mathcal{A}$):
\begin{enumerate}
\item $\textbf{(}ad(x)\otimes id + (-1)^{{|x|}{|r_{1}|}} id \otimes ad(x)\textbf{)}(r+\tau(r))=0$  $\Longleftrightarrow$ $\delta(x)+\tau(\delta(x))=0,$
\item $ad(x)\textbf{(}[r_{12}, r_{13}]+[r_{12}, r_{23}]+[r_{13}, r_{23}]\textbf{)}=0.$ \\
\end{enumerate}

Let $(\mathcal{A}, \bullet)$ be a commutative associative superalgebra and $r\in\mathcal{A}\otimes\mathcal{A}$. The linear map $\Delta$ defined by Eq. (\ref{cobord 2}) and $|r|=\bar{0}$ makes $(\mathcal{A}, \Delta)$ into a cocommutative and coassociative supercoalgebra if and only if the following conditions are satisfied (for any $x\in\mathcal{A}$):
\begin{enumerate}
\item $(L_{\bullet}(x)\otimes id-(-1)^{{|x|}{|r_{1}|}}id\otimes L_{\bullet}(x))(r+\tau(r))=0$ $\Longleftrightarrow$ $\Delta(x)-\tau(\Delta(x))=0,$
\item $(L_{\bullet}(x)\otimes id\otimes id-id\otimes id\otimes T_{\bullet}(x))(r_{13}\bullet r_{12}-r_{12}\bullet r_{23}+r_{23}\bullet r_{13})=0,$
\end{enumerate}
where $T_{\bullet}(x)y=y\bullet x.$\\
\ \ \ \ Next we introduce the notion of coboundary Poisson superbialgebra.\\

\begin{df} A coboundary Poisson superbialgebra $(\mathcal{A}, \{,\}, \bullet, \delta, \Delta, r)$ consists of a Poisson superbialgebra $(\mathcal{A}, \{,\}, \bullet, \delta, \Delta)$ and an element $r=\sum r_{1}\otimes r_{2}\in\mathcal{A}\otimes\mathcal{A}$ and $|r|=\bar{0}$
such that
\begin{eqnarray}\label{Poisson 01}
\delta(x)=(ad_{\{,\}}(x)\otimes id +
(-1)^{{|x|}{|r_{1}|}}id\otimes ad_{\{,\}}(x))r,
\end{eqnarray}
\begin{eqnarray}\label{Poisson 02}
\ \ \ \ \ \ \ \ \ \ \ \ \ \ \ \ \ \ \ \ \ \ \  \Delta(x)=(L_{\bullet}(x)\otimes id-(-1)^{{|x|}{|r_{1}|}}id\otimes L_{\bullet}(x))r, \ \ \ \ \ \ \ \ for \ all \ x \ \in \mathcal{A}.
\end{eqnarray}

\end{df}
Obviously, a coboundary Poisson superbialgebra $(\mathcal{A}, \{,\}, \bullet, \delta, \Delta, r)$ is equivalent to the fact that both $(\mathcal{A}, \{,\}, \delta, r)$ (as a Lie superbialgebra) and $(\mathcal{A}, \bullet, \Delta, r)$ (as an infinitesimal superbialgebra) are coboundary.

Now we consider  when $(\mathcal{A}, \delta, \Delta, \alpha)$ becomes a Poisson supercoalgebra, where $\delta$ and $\Delta$ are defined by
Eqs. (\ref{Poisson 01}) and (\ref{Poisson 02}) for some $r\in\mathcal{A}\otimes\mathcal{A}$, respectively. Let $\mathcal{A}$ be a $\mathbb{K}$-vector space equipped with two comultiplications $\delta, \ \Delta: \mathcal{A}\rightarrow \mathcal{A}\otimes\mathcal{A}$. Then $(\mathcal{A}, \delta, \Delta)$ becomes a Poisson supercoalgebra for $\delta$ corresponding to the Lie cobracket and $\Delta$ corresponding to the coproduct of the cocommutative supercoalgebra if and only if $(\mathcal{A}, \delta)$ is a Lie supercoalgebra and  $(\mathcal{A}, \Delta)$ is a cocommutative supercoalgebra and
\begin{eqnarray}
W(x)=(id\otimes\Delta)\delta(x)-(\delta\otimes id)\Delta(x)-(\tau\otimes id)(id\otimes\delta)\Delta(x)=0, \ \ \ \ \ \ for \ all \ x \ \in \mathcal{A}.
\end{eqnarray}

Let $(\mathcal{A}, \{,\}, \bullet)$ be a Poisson superalgebra. Define two comultiplications $\delta, \ \Delta: \mathcal{A}\rightarrow \mathcal{A}\otimes\mathcal{A}$ by Eqs. (\ref{Poisson 01}) and (\ref{Poisson 02}) with $r=\sum r_{1}\otimes r_{2}\in\mathcal{A}\otimes\mathcal{A}$ and $|r|=\bar{0}$, respectively. Then in this case (for any $x\in\mathcal{A}$)

\begin{align}\label{w(x)}
W(x)&=-(ad_{\{,\}}(x)\otimes id\otimes id)A(r)
+(-1)^{{|x|}{|r_{1}|}}(id\otimes L_{\bullet}(x)\otimes id-id\otimes id\otimes L_{\bullet}(x))C(r)\\&
-\sum \textbf{(}(-1)^{{|x|}{|r_{1}|}}(ad_{\{,\}}(r_{1})\otimes id)(L_{\bullet}(x)\otimes id-id\otimes L_{\bullet}(x))(r+\tau(r))\textbf{)}\otimes r_{2}.
\nonumber\end{align}

In fact, it follows from
\begin{align*}
W(x)&=\sum \{x, r'_{1}\}\otimes r'_{2}\bullet r_{1}\otimes r_{2}-(-1)^{{|r_{1}|}{|r_{2}|}}\{x, r'_{1}\}\otimes r_{1}\otimes r'_{2}\bullet r_{2}-\{x\bullet r'_{1}, r_{1}\}\otimes r_{2}\otimes r'_{2}\\
&+ \{x\bullet r'_{2}, r_{1}\}\otimes r'_{1}\otimes r_{2}+(-1)^{{|x|}{|r_{1}|}}r'_{1}\otimes\{x, r'_{2}\}\bullet r_{1}\otimes r_{2}-(-1)^{{|x|}{|r_{1}|}+{|r_{1}|}{|r_{2}|}}r_{1}\otimes\{x\bullet r'_{1}, r_{2}\}\otimes r'_{2}\\
&-\{r'_{2}, r_{1}\}\otimes x\bullet r'_{1}\otimes r_{2}-(-1)^{{|x|}{|r_{1}|}}r_{1}\otimes x\bullet r'_{1}\otimes \{r'_{2}, r_{2}\}-(-1)^{{|r_{1}|}{|r_{2}|}}r'_{1}\otimes r_{1}\otimes\{x, r'_{2}\}\bullet r_{2}\\
&+(-1)^{{|x|}{|r_{1}|}}\{r'_{1}, r_{1}\}\otimes r_{2}\otimes x\bullet r'_{2}+(-1)^{{|x|}{|r_{1}|}+{|r_{1}|}{|r_{2}|}}r_{1}\otimes\{r'_{1}, r_{2}\}\otimes x\bullet r'_{2}+r_{1}\otimes r'_{1}\otimes\{x\bullet r'_{2}, r_{2}\},
\end{align*}

in which the sum of the first four terms is $$-(ad_{\{,\}}(x)\otimes id\otimes id)A(r)-\sum \textbf{(}((-1)^{{|x|}{|r_{1}|}}ad_{\{,\}}(r_{1})\otimes id)(L_{\bullet}(x)\otimes id)(r+\tau(r))\textbf{)}\otimes r_{2},$$
the sum of 5th to 8th is $$(-1)^{{|x|}{|r_{1}|}}(id\otimes L_{\bullet}(x)\otimes id)C(r)+\sum \textbf{(}(-1)^{{|x|}{|r_{1}|}}(ad_{\{,\}}(r_{1})\otimes id)(id\otimes L_{\bullet}(x))(r+\tau(r))\textbf{)}\otimes r_{2}$$ and the sum of the last four terms is $-(-1)^{{|x|}{|r_{1}|}}(id\otimes id\otimes L_{\bullet}(x))C(r)$. Therefore the following conclusion follows:\\
\begin{thm}\label{cobod Poisson} Let $(\mathcal{A}, \{,\}, \bullet)$ be a Poisson superalgebra and $r\in\mathcal{A}\otimes\mathcal{A}$ such that $|r|=\bar{0}$. Then the comultiplications $\delta$ and $\Delta$ defined by Eqs. (\ref{Poisson 01}) and (\ref{Poisson 02}), respectively, make $(\mathcal{A}, \delta, \Delta)$ into a Poisson supercoalgebra such that $(\mathcal{A}, \{,\}, \bullet, \delta, \Delta, \alpha)$ is a Poisson superbialgebra if and only if the following conditions are satisfied (for any $x\in\mathcal{A}$):
\begin{enumerate}
\item $(ad_{\{,\}}(x)\otimes id+(-1)^{{|x|}{|r_{1}|}}id\otimes ad_{\{,\}}(x))(r+\tau(r))=(L_{\bullet}(x)\otimes id-(-1)^{{|x|}{|r_{1}|}}id\otimes L_{\bullet}(x))(r+\tau(r))=0,$
\item $(L_{\bullet}(x)\otimes id\otimes id-id\otimes id\otimes T_{\bullet}(x))A(r)=0,$
\item $ad(x)(C(r))=0,$
\item $W(x)=0,$ where $W(x)$ is given by Eq. (\ref{w(x)}) and $r=\sum r_{1}\otimes r_{2}$.
\end{enumerate}
\end{thm}

Let $(\mathcal{A}, \{,\}, \bullet)$ be a Poisson superalgebra and $r\in\mathcal{A}\otimes\mathcal{A}$ such that $|r|=\bar{0}$. We say $r$ satisfies Poisson Yang-Baxter equation (PYBE) if $r$ satisfies both CYBE and AYBE. Therefore a direct consequence is that if $r$ is a skew-supersymmetric solution of PYBE in a Poisson superalgebra $(\mathcal{A}, \{,\}, \bullet)$, then the comultiplications $\delta$ and $\Delta$ defined by Eqs. (\ref{Poisson 01}) and (\ref{Poisson 02}), respectively, make $(\mathcal{A}, \{,\}, \bullet, \delta, \Delta)$ into a Poisson superbialgebra.\\

Another important consequence of Theorem \ref{cobod Poisson} is the following Poisson superalgebra analogue of the Drinfeld double construction.
\begin{thm}\label{fin de cobord} Let $(\mathcal{A}, \{,\}_{1}, \bullet_{1}, \delta, \Delta)$ be a Poisson superbialgebra. Then there is a canonical coboundary Poisson superbialgebra structure on $\mathcal{A}\oplus\mathcal{A}^{\ast}$.

\end{thm}

\begin{proof} Let $r\in\mathcal{A}\otimes\mathcal{A}^{\ast}\subset(\mathcal{A}\oplus\mathcal{A}^{\ast})\otimes(\mathcal{A}\oplus\mathcal{A}^{\ast})$ correspond to the identity map $id:\mathcal{A}\rightarrow \mathcal{A}$. Let $\{e_{1},...,e_{n}\}$ be a basis of $\mathcal{A}$ and $\{e^{\ast}_{1},...,e^{\ast}_{n}\}$ be its dual basis. Then $r= \sum_{i} e_{i}\otimes e^{\ast}_{i}$ such that $|r|=\bar{0}$. Suppose that the Poisson superalgebra structure $(\{,\}, \bullet)$ on $\mathcal{A}\oplus\mathcal{A}^{\ast}$ is given by $\mathcal{PD}(\mathcal{A})=\mathcal{A}\bowtie_{{ad_{\{,\}_{1}}^{\ast}, -L^{\ast}_{\bullet_{1}}}}^{ad_{\{,\}_{2}}^{\ast}, -L^{\ast}_{\bullet_{2}}}\mathcal{A}^{\ast}$, where $(\{,\}_{_{2}}, \bullet_{2})$ is the Poisson superalgebra structure on $\mathcal{A}^{\ast}$ induced by $(\delta^{\ast}, \Delta^{\ast})$. Then we have, for any $x, y\in\mathcal{A}$, $a, b\in\mathcal{A}^{\ast}$, 
$$\{x, y\}=\{x, y\}_{1}, \ \  x\bullet y=x\bullet_{1} y,  \ \ \{a, b\}=\{a, b\}_{2}, \ \ a\bullet b=a\bullet_{2} b, \ \  \{x, a\}=-(-1)^{{|x|}{|a|}}\{a, x\}=ad_{\{,\}_{1}}^{\ast}(x)a$$ $-(-1)^{{|x|}{|a|}}ad_{\{,\}_{2}}^{\ast}(a)x,$ $x\bullet a=(-1)^{{|x|}{|a|}}a\bullet x=-L_{\bullet_{1}}^{\ast}(x)a-(-1)^{{|x|}{|a|}} L_{\bullet_{2}}^{\ast}(a)x$. It is straightforward to prove that $r$ satisfies CYBE and AYBE such that $|r|=\bar{0}$ and for any  $u\in\mathcal{PD}(\mathcal{A})$ $$(ad_{\{,\}}(u)\otimes id+(-1)^{{|u|}{|r_{1}|}}id\otimes ad_{\{,\}}(u))(r+\tau(r))
=(L_{\bullet}(u)\otimes id-(-1)^{{|u|}{|r_{1}|}}id\otimes L_{\bullet}(u))(r+\tau(r))=0.$$
So $r$ satisfies the conditions in Theorem \ref{cobod Poisson}. Thus, for any $u\in\mathcal{PD}(\mathcal{A})$
$$\delta_{\mathcal{PD}}(u)=(ad_{\{,\}}(u)\otimes id+(-1)^{{|u|}{|r_{1}|}}id\otimes ad(u))r, \ \ and \ \ \Delta_{\mathcal{PD}}(u)(L_{\bullet}(u)\otimes id
-(-1)^{{|u|}{|r_{1}|}}id\otimes L_{\bullet}(u))r,$$ induce a coboundary Poisson superbialgebra structure on $\mathcal{PD}(\mathcal{A})$.
\end{proof}

Let $(\mathcal{A}, \{,\}_{1}, \bullet_{1}, \delta_{1}, \Delta_{1})$ be a Poisson superbialgebra. With the Poisson superbialgebra structure given in
Theorem \ref{fin de cobord}, $\mathcal{A}\oplus\mathcal{A}^{\ast}$ is called the Drinfeld classical double of $\mathcal{A}$. As in the proof of Theorem \ref{fin de cobord}, we denote it by $\mathcal{PD}(\mathcal{A})$.

\section{$\mathcal{O}$-operators of Poisson superalgebras, Post-Poisson superalgebras,
and quasitriangular Poisson superbialgebras}

In this section we introduce the notions of $\mathcal{O}$-operator of weight $\lambda\in \mathbb{K}$ on a Poisson superalgebra, post-
Poisson superalgebra, and quasitriangular Poisson superbialgebras. We use $\mathcal{O}$-operators on Poisson superalgebras to
construct post-Poisson superalgebras. We show that a quasitriangular Poisson superbialgebra naturally gives a
post-Poisson superalgebra.

\begin{df} Let $(\mathcal{A},[\cdot,\cdot]_{\mathcal{A}})$ and $(\mathcal{A}',[\cdot,\cdot]_{\mathcal{A}'})$ be two Lie superalgebras. Suppose that $\omega$ is a Lie superalgebra homomorphism from $\mathcal{A}$ to $Der_{\mathbb{K}}(\mathcal{A}')$ the Lie superalgebra consisting of all the derivations of $\mathcal{A}'$. Then $(\mathcal{A}',[\cdot,\cdot]_{\mathcal{A}'}, \omega)$ is called a $\mathcal{A}$-Lie superalgebra.
\end{df}

\begin{df} Let $(\mathcal{A}, \bullet)$ and $(R, \cdot)$ be two commutative associative superalgebras. Let $\varphi:\mathcal{A}\longrightarrow End(R)$ be an even linear map. Then the triple $(R, \cdot, \varphi)$ is called an $\mathcal{A}$-module superalgebra if
\begin{eqnarray}
\varphi(x\bullet y)=\varphi(x)\varphi(y), \ \ \
\varphi(x)(u\cdot v)=(\varphi(x)u)\cdot v,
\end{eqnarray}
for $x, y \in\mathcal{A}$, $u, v\in R$.
\end{df}

\begin{df} Let $(\mathcal{P},\{,\},\bullet)$ and $(V,\{,\}_1,\bullet_1)$ be two Poisson superalgebras. Let $\psi_{\{,\}},\psi_\bullet:\mathcal{P}\longrightarrow End(R)$ be two even linear maps such that
\begin{enumerate}
\item $(V,\{,\}_1,\psi_{\{,\}})$ is a $\mathcal{P}$-Lie superalgebra, where $\mathcal{P}$ is seen as a Lie superalgebra with respect to the Lie bracket $\{,\}$.
\item $(V,\bullet_1,\psi_\bullet)$ is a $\mathcal{P}$-module superalgebra, where $\mathcal{P}$ is seen as a commutative associative superalgebra with respect to the commutative associative product $\bullet$.
\item $(\psi_{\{,\}},\psi_\bullet, V)$ is a module of $\mathcal{P}$.
\item The following equations hold:
\end{enumerate}
\begin{eqnarray}
\psi_{\{,\}}(x)(u\bullet_{1}v)= (\psi_{\{,\}}(x)u)\bullet_{1} v+(-1)^{{|x|}{|u|}}u\bullet_{1}(\psi_{\{,\}}(x)v),~~\forall~~x\in \mathcal{P},~~ u,v\in V,
\end{eqnarray}
\begin{eqnarray}
\psi_{\bullet}(x)\{u,v\}_1= (\psi_{\{,\}}(x)u)\bullet_{1}v+(-1)^{{|x|}{|u|}}\{u,\psi_{\bullet}(x)v\}_1, ~~\forall~~x\in \mathcal{P},~~ u,v\in V.
\end{eqnarray}

Then $(V,\{,\}_1,\bullet_1,\psi_{\{,\}},\psi_{\bullet})$ is called a $\mathcal{P}$-module Poisson superalgebra.
\end{df}

\begin{prop}
With above notations, $(V, \{,\}_1, \bullet_1, \psi_{\{,\}}, \psi_{\bullet})$ is a $\mathcal{P}$-module Poisson superalgebra if and only if the direct sum of vector spaces $\mathcal{P}\oplus V$ is turned to a Poisson superalgebra with  the operations defined as
\begin{eqnarray}
\{(x,u),(y,v)\} = (\{x,y\},\psi_{\{,\}}(x)v-(-1)^{{|x|}{|y|}}\psi_{\{,\}}(y)u+\{u,v\}_1),
\end{eqnarray}
\begin{eqnarray}
(x,u)\bullet(y,v)=(x\bullet y,\psi_{\bullet}(x)v+(-1)^{{|x|}{|y|}}\psi_{\bullet}(y)u+u\bullet_{1} v),
\end{eqnarray}
for any $x,y\in \mathcal{P},~~ u,v\in V$.
\end{prop}

\begin{proof} Straightly from Proposition \ref{S()YYY}.
\end{proof}
If $(\mathcal{P},\{,\},\bullet)$ is a Poisson superalgebra, then it is easy to see that $(\mathcal{P},\{,\},\bullet,ad_{\{,\}},L_\bullet)$ is a $\mathcal{P}$-module Poisson superalgebra.

\begin{df} A (left) post-Lie superalgebra is a superspace $\mathcal{A}$ with two even bilinear operations
$([\cdot ,\cdot ], \diamond)$ satisfying
the following equations: 
\begin{eqnarray}
 [x,y]= -(-1)^{{|x|}{|y|}}[y,x],
\end{eqnarray}
\begin{eqnarray}
(-1)^{{|x|}{|z|}}[x,[y,z]]+(-1)^{{|z|}{|y|}}[z,[x,y]]
+(-1)^{{|y|}{|x|}}[y,[z,x]]=0,
\end{eqnarray}
\begin{eqnarray}
(-1)^{{|z|}{|y|}} z\diamond (y\diamond x)- y\diamond (z\diamond x)+ (y\diamond z)\diamond x - (-1)^{{|z|}{|y|}} (z \diamond y)\diamond x +[y,z]\diamond x = 0,
\end{eqnarray}
\begin{eqnarray}
z\diamond[x,y]-[z\diamond x,y]-(-1)^{{|z|}{|x|}}[x,z\diamond y]= 0,
\end{eqnarray}
for any elements $x,y,z$ in $\mathcal{A}$.
\end{df}

\begin{df} A (left) commutative dendriform supertrialgebra is a superspace $\mathcal{A}$ equipped with two even bilinear operations
$(\cdot, \succ):\mathcal{A}\otimes \mathcal{A}\longrightarrow \mathcal{A}$ satisfying
the following equations:
\begin{eqnarray}
x\cdot y= (-1)^{{|x|}{|y|}} y\cdot x,
\end{eqnarray}
\begin{eqnarray}
(x\cdot y)\cdot z = x\cdot (y\cdot z),
\end{eqnarray}
\begin{eqnarray}
x\succ (y\succ z)=\Big(x\succ y+(-1)^{{|x|}{|y|}} y\succ x+x\cdot y\Big)\succ z,
\end{eqnarray}
\begin{eqnarray}
(x\succ y)\cdot z=x\succ (y\cdot z),
\end{eqnarray}
 for any elements $x,y,z$ in $\mathcal{A}$.
\end{df}

\begin{df} A (left) post-Poisson superalgebra is a superspace $\mathcal{A}$ equipped with four even bilinear operations $([\cdot ,\cdot ], \diamond, \cdot, \succ)$ such that $(\mathcal{A}, [\cdot ,\cdot ], \diamond)$ is a (left) post-Lie superalgebra, $(\mathcal{A}, \cdot, \succ)$ is a (left)
commutative dendriform supertrialgebra, and they are compatible in the sense that, for any elements $x,y,z$ in $\mathcal{A}$, 
\begin{eqnarray}
[x, y \cdot z]=[x, y]\cdot z+ (-1)^{{|x|}{|y|}}y\cdot [x, z],
\end{eqnarray}
\begin{eqnarray}
[x, z\succ y]=(-1)^{{|x|}{|z|}}z\succ[x, y]-(-1)^{{|y|}{|z|}+{|x|}{|y|}+{|x|}{|z|}}y\cdot(z\diamond x),
\end{eqnarray}
\begin{eqnarray}
x\diamond(y\cdot z)=(x\diamond y)\cdot z+(-1)^{{|x|}{|y|}}y\cdot(x\diamond z),
\end{eqnarray}
\begin{eqnarray}
\Big(y\succ z+ (-1)^{{|y|}{|z|}}z\succ y+y\cdot z\Big)\diamond x=(-1)^{{|y|}{|z|}}z\succ(y\diamond x)+y\succ(z\diamond x),
\end{eqnarray}
\begin{eqnarray}
x\diamond(z\succ y)=(-1)^{{|x|}{|z|}}z\succ(x\diamond y)+\Big(x\diamond z-(-1)^{{|x|}{|z|}}z\diamond x+[x, z] \Big)\succ y.
\end{eqnarray}
\end{df}

\begin{thm} Let $(\mathcal{A}, [\cdot ,\cdot ], \diamond, \cdot, \succ)$ be a post-Poisson superalgebra. Define two new even bilinear map
$\{,\},\bullet:\mathcal{A}\otimes \mathcal{A}\longrightarrow \mathcal{A}$ on $\mathcal{A}$  by
\begin{eqnarray}
\ \ \ \ \ \ \ \ \ \ \{x,y\}= x\diamond y-(-1)^{{|x|}{|y|}}y\diamond x+[x,y], \ \ \ \ \
x \bullet y= x\succ y+(-1)^{{|x|}{|y|}}y\succ x+ x\cdot y, \ \ \ \forall x,y\in\mathcal{A}.
\end{eqnarray}
Then $(\mathcal{A}, \{,\}, \bullet)$ becomes a Poisson superalgebra.\\ It is called the associated Poisson superalgebra of $(\mathcal{A}, [\cdot ,\cdot ], \diamond, \cdot, \succ)$ and is denoted by $(P(\mathcal{A}), \{,\}, \bullet)$.\\ 
Moreover, $(\mathcal{A}, [\cdot ,\cdot ],\cdot, L_\diamond, L_\succ)$ is a $P$-module Poisson superalgebra of $(P(\mathcal{A}), \{,\}, \bullet)$.
\end{thm}

\begin{proof} For any $x, y, z\in\mathcal{A}$, we will check that  $(\mathcal{A}, \{,\})$ is a Lie superalgebra. In fact, we have
\begin{align*}
\{x,y\}&= x\diamond y-(-1)^{{|x|}{|y|}}y\diamond x+[x,y]\\
&=-(-1)^{{|x|}{|y|}}y\diamond x+x\diamond y-(-1)^{{|x|}{|y|}}[y,x]\\
&=-(-1)^{{|x|}{|y|}}(y\diamond x-(-1)^{{|x|}{|y|}}x\diamond y+[y,x])\\
&=-(-1)^{{|x|}{|y|}}\{y,x\}.
\end{align*}
So the skew-supersymmetry holds. For the superJacobi identity, we have
\begin{align*}
&(-1)^{{|x|}{|z|}}\{x, \{y, z\}\}+(-1)^{{|y|}{|x|}}\{y, \{z, x\}\}+(-1)^{{|z|}{|y|}}\{z, \{x, y\}\}\\
&=(-1)^{{|x|}{|y|}}\{x, y\diamond z-(-1)^{{|y|}{|z|}}z\diamond y+[y, z]\}+
(-1)^{{|y|}{|x|}}\{y, z\diamond x-(-1)^{{|z|}{|x|}}z\diamond x+[z, x]\}\\
&+(-1)^{{|z|}{|y|}}\{z, x\diamond y-(-1)^{{|x|}{|y|}}y\diamond x+[x, y]\}\\
&=(-1)^{{|x|}{|z|}}x\diamond(y\diamond z)-(-1)^{{|x|}{|z|}+{|y|}{|z|}}x\diamond(z\diamond y)+(-1)^{{|x|}{|z|}}x\diamond[y,z]-(-1)^{{|x|}{|y|}}(y\diamond z)\diamond x\\
&+(-1)^{{|x|}{|y|}+{|y|}{|z|}}(z\diamond y)\diamond x-(-1)^{{|x|}{|y|}}[y,z]\diamond x+(-1)^{{|x|}{|z|}}[x, y\diamond z]-(-1)^{{|x|}{|z|}+{|y|}{|z|}}[x, z\diamond y]\\
&+(-1)^{{|x|}{|z|}}[x, [y, z]]+(-1)^{{|x|}{|y|}}y\diamond(z\diamond x)-(-1)^{{|x|}{|z|}+{|x|}{|y|}}y\diamond(x\diamond z)+(-1)^{{|x|}{|y|}}y\diamond[z,x]\\
&-(-1)^{{|y|}{|z|}}(z\diamond x)\diamond y+(-1)^{{|x|}{|z|}+{|y|}{|z|}}(x\diamond z)\diamond y-(-1)^{{|y|}{|z|}}[z,x]\diamond y+(-1)^{{|x|}{|y|}}[y, z\diamond x]\\
&-(-1)^{{|x|}{|y|}+{|x|}{|z|}}[y, x\diamond z]+(-1)^{{|x|}{|y|}}[y, [z, x]]+(-1)^{{|z|}{|y|}}z\diamond(x\diamond y)-(-1)^{{|x|}{|y|}+{|z|}{|y|}}z\diamond(y\diamond x)\\
&+(-1)^{{|y|}{|z|}}z\diamond [x, y]-(-1)^{{|x|}{|z|}}(x\diamond y)\diamond z+(-1)^{{|x|}{|y|}+{|x|}{|z|}}(y\diamond x)\diamond z-(-1)^{{|x|}{|z|}}[x, y]\diamond z\\
&+(-1)^{{|z|}{|y|}}[z, x\diamond y]-(-1)^{{|z|}{|y|}+{|x|}{|y|}}[z, y\diamond x]+(-1)^{{|z|}{|y|}}[z, [x, y]]
=0.
\end{align*}
Next, we check that $(\mathcal{A}, \bullet)$ is a commutative associative superalgebra. In fact, for any $x, y, z\in\mathcal{A}$, also one may check directly that :
$x\bullet y=x\succ y+(-1)^{{|x|}{|y|}}y\succ x+ x\cdot y=(-1)^{{|x|}{|y|}}(y\succ x+(-1)^{{|x|}{|y|}}x\succ y+y\cdot x)=(-1)^{{|x|}{|y|}}y\bullet x.$\\
Now we check the following equality for any $x, y, z\in\mathcal{A}$, $x\bullet(y\bullet z)=(x\bullet y)\bullet z$. In fact, we have
\begin{align*}
&x\bullet(y\bullet z)=x\bullet(y\succ z+(-1)^{{|y|}{|z|}}z\succ y+ y\cdot z)\\
&=x\succ\Big(y\succ z+ (-1)^{{|y|}{|z|}}z\succ y+y\cdot z\Big)+(-1)^{{|x|}{|y|}+{|x|}{|z|}}\Big(y\succ z+ (-1)^{{|y|}{|z|}}z\succ y+y\cdot z\Big)\succ x\\
&+x\cdot\Big(y\succ z+ (-1)^{{|y|}{|z|}}z\succ y+y\cdot z\Big)\\
&=x\succ(y\succ z)+(-1)^{{|y|}{|z|}}x\succ(z\succ y)+x\succ(y\cdot z)+(-1)^{{|x|}{|y|}+{|x|}{|z|}}y\succ(z\succ x)+x\cdot(y\succ z)\\
&+(-1)^{{|y|}{|z|}}x\cdot(z\succ y)+x\cdot(y\cdot z)\\
&=x\succ(y\succ z)+(-1)^{{|x|}{|z|}+{|y|}{|z|}}z\succ(x\succ y)+(x\succ y)\cdot z+(-1)^{{|x|}{|z|}+{|y|}{|z|}+{|x|}{|y|}}z\succ(y\succ x)\\
&+(-1)^{{|x|}{|y|}}(y\succ x)\cdot z+(-1)^{{|x|}{|z|}+{|y|}{|z|}}z\succ(y\cdot z)+(x\cdot y)\cdot z\\
&=(x\bullet y)\bullet z.
\end{align*}
Finally, we check the condition :\ \ \
$\{x, y\bullet z\}=\{x, y\}\bullet z+(-1)^{{|x|}{|y|}}y\bullet\{x, z\}.$ In fact, we have
\begin{align*}
&\{x, y\bullet z\}=\{x, y\succ z+(-1)^{{|y|}{|z|}}z\succ y+ y\cdot z\}=x\diamond\Big(y\succ z+ (-1)^{{|y|}{|z|}}z\succ y+y\cdot z\Big)\\
&-(-1)^{{|x|}{|y|}+{|x|}{|z|}}\Big(y\succ z+ (-1)^{{|y|}{|z|}}z\succ y+y\cdot z\Big)\diamond x+[x, y\succ z+(-1)^{{|y|}{|z|}}z\succ y+ y\cdot z]\\
&=x\diamond(y\succ z)+(-1)^{{|y|}{|z|}}x\diamond(z\succ y)+x\diamond(y\cdot z)-(-1)^{{|x|}{|y|}+{|x|}{|z|}+{|y|}{|z|}}z\succ(y\diamond x)-(-1)^{{|x|}{|y|}+{|x|}{|z|}}y\succ(z\diamond x)\\
&+[x, y\succ z]+(-1)^{{|y|}{|z|}}[x, z\succ y]+[x, y\cdot z]\\
&=(x\diamond y)\succ z-(-1)^{{|x|}{|y|}}(y\diamond x)\succ z+[x, y]\succ z+(-1)^{{|x|}{|z|}+{|y|}{|z|}}z\succ(x\diamond y)-(-1)^{{|x|}{|z|}+{|x|}{|y|}+{|y|}{|z|}}z\succ(y\diamond x)\\
&+(-1)^{{|x|}{|z|}+{|y|}{|z|}}z\succ[x, y]+(x\diamond y)\cdot z-(-1)^{{|x|}{|y|}}(y\diamond x)\cdot z+[x, y]\cdot z+(-1)^{{|x|}{|y|}}y\succ(x\diamond z)\\
&-(-1)^{{|x|}{|y|}+{|x|}{|z|}}y\succ(z\diamond x)+(-1)^{{|x|}{|y|}}y\succ[x, z]+(-1)^{{|y|}{|z|}}(x\diamond z)\succ y-(-1)^{{|y|}{|z|}+{|x|}{|z|}}(z\diamond x)\succ y\\
&+(-1)^{{|y|}{|z|}}[x, z]\succ y+(-1)^{{|x|}{|y|}}y\cdot(x\diamond z)-(-1)^{{|x|}{|y|}+{|x|}{|z|}}y\cdot(z\diamond x)+(-1)^{{|x|}{|y|}}y\cdot[x, z]\\
&=\{x, y\}\bullet z+(-1)^{{|x|}{|y|}}y\bullet\{x, z\},
\end{align*}
as desired. Thus $(\mathcal{A}, \{,\}, \bullet)$ is a Poisson superalgebra. 
\end{proof}

\begin{df} Let $(\mathcal{A}, \{,\}, \bullet)$ be a Poisson superalgebra and $(V, \{,\}_1, \bullet_{1}, \psi_{\{,\}}, \psi_\bullet)$ be an $\mathcal{A}$-module Poisson superalgebra. An even linear map $T:V \longrightarrow \mathcal{A}$ is called an $\mathcal{O}$-operator of weight $\lambda\in \mathbb{K}$ associated to $(V,\{,\}_1, \bullet_{1},\psi_{\{,\}}, \psi_\bullet)$ if for any $u, v\in V$
\begin{eqnarray}\label{fin super 1}
\{T(u),T(v)\}&=& T\Big(\psi_{\{,\}}(T(u))v-(-1)^{{|u|}{|v|}}\psi_{\{,\}}(T(v))u +\lambda\{u,v\}_1 \Big),
\end{eqnarray}
\begin{eqnarray}\label{fin super 2}
T(u) \bullet T(v)&=& T\Big(\psi_\bullet(T(u))v+(-1)^{{|u|}{|v|}}\psi_\bullet(T(v))u+\lambda u\bullet_{1}v \Big).
\end{eqnarray}
When $(V, \{,\}_1, \bullet_{1}, \psi_{\{,\}}, \psi_\bullet)
=(\mathcal{A}, \{,\}, \bullet, ad_{\{,\}}, L_\bullet)$, Eqs (\ref{fin super 1}) and (\ref{fin super 2}) become
\begin{eqnarray}\label{fin superrrr 1}
\{T(u),T(v)\}= T\Big(\{T(u),v \}+\{u, T(v)\} +\lambda\{u,v\} \Big),
\end{eqnarray}
\begin{eqnarray}\label{fin superrrr 2}
T(u) \bullet T(v)&=& T\Big(T(u)\bullet v+u\bullet T(v)+\lambda u\bullet v \Big),
\end{eqnarray}
respectively. Equations (\ref{fin superrrr 1}) and (\ref{fin superrrr 2}) imply that $T:\mathcal{A}\longrightarrow \mathcal{A}$ is  a Rota-Baxter operator of weight $\lambda\in \mathbb{K}$ on the Lie superalgebra $(\mathcal{A},\{,\})$ and on the commutative associative superalgebra $(\mathcal{A},\bullet)$, respectively.
\end{df}

\begin{thm} Let $(\mathcal{A}, \{,\}, \bullet)$ be a Poisson superalgebra and $(V, \{,\}_1, \bullet_{1}, \psi_{\{,\}}, \psi_\bullet)$ be an $\mathcal{A}$-module Poisson superalgebra. Let the  even linear map $T:V \longrightarrow \mathcal{A}$ be an $\mathcal{O}$-operator of weight $\lambda\in \mathbb{K}$ associated to $(V,\{,\}_1, \bullet_{1},\psi_{\{,\}}, \psi_\bullet)$. Define four new even bilinear operations $[\cdot ,\cdot ], \diamond, \cdot, \succ: V\otimes V \longrightarrow V$ on $V$ as follows:
\begin{eqnarray}
\ \ \ \ \ \ [u, v]=\lambda\{u, v\}_{1}, \ \ \  u\diamond v=\psi_{\{,\}}(T(u))v, \ \ \ u\cdot v=\lambda u\bullet_{1} v \ \ \ u\succ v=\psi_\bullet(T(u))v, \ \ \ \forall u, v\in V.
\end{eqnarray}
Then $(V, [\cdot ,\cdot ], \diamond, \cdot, \succ)$ is a post-Poisson superalgebra and $T$ is a homomorphism of Poisson superalgebras
from the associated Poisson superalgebra $P(V)$ of $(V, [\cdot ,\cdot ], \diamond, \cdot, \succ)$ to $(\mathcal{A}, \{,\}, \bullet)$.
\end{thm}

\begin{proof} First we check that $(V, [\cdot ,\cdot ], \diamond)$ is a post-Lie superalgebra. For any $u, v, w\in V$, it is easy to obtain 
\begin{align*}
&[u, v]=-(-1)^{{|u|}{|v|}}[v, u],\\
&(-1)^{{|u|}{|w|}}[u,[v,w]]+(-1)^{{|w|}{|v|}}[w,[u,v]]
+(-1)^{{|v|}{|u|}}[v,[w,u]]=0.
\end{align*}
So it is sufficient to verify the following conditions:
\begin{align*}
&(-1)^{{|w|}{|v|}} w\diamond (v\diamond u)- v\diamond (w\diamond u)+ (v\diamond w)\diamond u - (-1)^{{|w|}{|v|}} (w \diamond v)\diamond u +[v,w]\diamond u = 0,\\
&w\diamond[u,v]-[w\diamond u,v]-(-1)^{{|w|}{|u|}}[u,w\diamond v]= 0.
\end{align*}
In fact, we have
\begin{align*}
&(-1)^{{|w|}{|v|}} w\diamond (v\diamond u)- v\diamond (w\diamond u)+ (v\diamond w)\diamond u - (-1)^{{|w|}{|v|}}(w \diamond v)\diamond u +[v,w]\diamond u = 0,\\
&=(-1)^{{|w|}{|v|}}\psi_{\{,\}}(T(w))\psi_{\{,\}}(T(v))u
-\psi_{\{,\}}(T(v))\psi_{\{,\}}(T(w))u+\psi_{\{,\}}(T(\psi_{\{,\}}(T(v))w))u\\
&-(-1)^{{|w|}{|v|}}\psi_{\{,\}}(T(\psi_{\{,\}}(T(w))v))u
+\psi_{\{,\}}(T(\lambda\{v, w\}_{1}))u\\
&=(-1)^{{|w|}{|v|}}\psi_{\{,\}}(T(w))\psi_{\{,\}}(T(v))u
-\psi_{\{,\}}(T(v))\psi_{\{,\}}(T(w))u+\psi_{\{,\}}(\{T(v),T(w)\})u\\
&=0,
\end{align*}
and
\begin{align*}
&w\diamond[u,v]-[w\diamond u,v]-(-1)^{{|w|}{|u|}}[u,w\diamond v]\\
&=\lambda\Big(w\diamond \{u, v\}_{1}-\{w\diamond u, v\}_{1}-(-1)^{{|w|}{|u|}}\{u,w\diamond v\}_{1}\Big)\\
&=\lambda\Big(\psi_{\{,\}}(T(w))\{u, v\}_{1}-\{\psi_{\{,\}}(T(w))u, v\}_{1}-(-1)^{{|w|}{|u|}}\{u,\psi_{\{,\}}(T(w))v\}_{1}\Big)\\
&=0,
\end{align*}
as required.\\
Next we will check that $(V,\cdot, \succ)$ is a commutative dendriform supertrialgebra.  For any $u, v, w\in V$, it is easy to obtain that $u\cdot v= (-1)^{{|u|}{|v|}} v\cdot u$, and
$(u\cdot v)\cdot w = u\cdot (v\cdot w)$. So we need to check the following conditions:
\begin{align*}
&u\succ (v\succ w)=\Big(u\succ v+(-1)^{{|u|}{|v|}} v\succ u+u\cdot v\Big)\succ w,\\
&(u\succ v)\cdot w=u\succ (v\cdot w).
\end{align*}
For the first equality, we calculate
\begin{align*}
&\Big(u\succ v+(-1)^{{|u|}{|v|}} v\succ u+u\cdot v\Big)\succ w\\
&=\Big(\psi_\bullet(T(u))v+(-1)^{{|u|}{|v|}}\psi_\bullet(T(v))u+\lambda u\bullet_{1} v\Big)\succ w\\
&=\psi_\bullet\Big(T\Big(\psi_\bullet(T(u))v
+(-1)^{{|u|}{|v|}}\psi_\bullet(T(v))u+\lambda u\bullet_{1} v)\Big)\Big)w\\
&=\psi_\bullet(T(u))\psi_\bullet(T(v))w\\
&=u\succ (v\succ w).
\end{align*}
For the second equality, we have
\begin{align*}
(u\cdot v)\cdot w=\lambda\psi_\bullet(T(u))v\bullet_{1} w=\psi_\bullet(T(u))(v\cdot w)=u\succ (v\cdot w).
\end{align*}
The proof of other cases is similar. Then $(V, [\cdot ,\cdot ], \diamond, \cdot, \succ)$ becomes a post-Poisson superalgebra and $T$ is a homomorphism of Poisson superalgebras from the associated Poisson superalgebra $P(V)$ of $(V, [\cdot ,\cdot ], \diamond, \cdot, \succ)$ to $(\mathcal{A}, \{,\}, \bullet)$.

\end{proof}

\begin{df} A coboundary Poisson superbialgebra $(\mathcal{A}, \{,\}, \bullet, \delta, \Delta, r)$ is called quasitriangular if $r$ is a solution of PYBE, that is, it satisfies both CYBE and AYBE.
\end{df}
\begin{rem}
1) Obviously, a coboundary Poisson superbialgebra $(\mathcal{A}, \{,\}, \bullet, \delta, \Delta, r)$ is quasitriangular if and only if $(\mathcal{A}, \{,\}, \delta, r)$ as a coboundary Lie superbialgebra and $(\mathcal{A}, \bullet, \Delta, r)$ as a coboundary infinitesimal superbialgebra are both quasitriangular.\\
2) If $(\mathcal{A}, \{,\}, \bullet, \delta, \Delta, r)$ is a quasitriangular Poisson superbialgebra. Then by Theorem \ref{cobod Poisson}, $r$ also satisfies the
following equations:
\begin{align}
(ad_{\{,\}}(x)\otimes id+(-1)^{{|x|}{|r_{1}|}}id\otimes ad_{\{,\}}(x))(r+\tau(r))=(L_{\bullet}(x)\otimes id-(-1)^{{|x|}{|r_{1}|}}id\otimes L_{\bullet}(x))(r+\tau(r))=0,
\end{align}
 for  all  $x\in \mathcal{A}$.
\end{rem}

Let $V=V_{\bar{0}}\oplus V_{\bar{1}}$ be a finite-dimensional superspace and $r\in V\otimes V$. Then $r$ can be identified as an even linear map from $V^{\ast}$ to $V$ which we still denote by $r$ through
\begin{eqnarray}\label{fin cccb}
\langle b, r(a)\rangle=\langle a\otimes b, r\rangle,  \ \ \ \forall \  a, b\in V^{\ast}.
\end{eqnarray}
Define a even linear map $r':V^{\ast} \longrightarrow V$ by
\begin{eqnarray}
\langle a, r'(b)\rangle=\langle r, a\otimes b\rangle,  \ \ \ \forall \  a, b\in V^{\ast}.
\end{eqnarray}

We call
\begin{eqnarray}\label{fin super}
\alpha=\alpha_{r}=\frac{r-r'}{2}, \ \ \ \ \ \beta=\beta_{r}=\frac{r+r'}{2},
\end{eqnarray}

the skew-supersymmetry part and the supersymmetric part of $r$ respectively. Therefore, we have the following theorem.

\begin{thm}\label{quasitriang} Let $(\mathcal{P}, \{,\}, \bullet, \delta, \Delta, r)$ be a quasitriangular Poisson superbialgebra. Let the even linear map $\beta$ be
defined by Eq. \ref{fin super}. Define two new even bilinear operations $[\cdot ,\cdot ], \cdot : \mathcal{P}\otimes \mathcal{P} \longrightarrow \mathcal{P}$ as follows:
\begin{eqnarray}
[a, b]=-2ad_{\{,\}}^{\ast}(\beta(a))b, \ \ \ \ \ a\cdot b=2L^{\ast}_{\bullet}(\beta(a))b, \ \ \forall \  a, b\in \mathcal{P}^{\ast}.
\end{eqnarray}
Then $(\mathcal{P}^{\ast}, [\cdot ,\cdot ], \cdot,  ad_{\{,\}}^{\ast}, -L^{\ast}_{\bullet})$ becomes a $\mathcal{P}$-module Poisson superalgebra of $(\mathcal{A}, \{,\}, \bullet)$. Moreover, regarded as an even linear map from $\mathcal{P}^{\ast}$ to $\mathcal{P}$ through Eq. fin \ref{fin cccb} is an $\mathcal{O}$-operator of weight $1$ associated to $(\mathcal{P}^{\ast}, [\cdot ,\cdot ], \cdot,  ad_{\{,\}}^{\ast}, -L^{\ast}_{\bullet})$ that is,
\begin{eqnarray}
\{r(a), r(b)\}= r(ad_{\{,\}}^{\ast}(r(a))b-(-1)^{{|a|}{|b|}}ad_{\{,\}}^{\ast}(r(b))a+[a, b]), \ \ \ \ \ \forall \  a, b\in \mathcal{P}^{\ast},
\end{eqnarray}
\begin{eqnarray}
r(a) \bullet r(b)=r(-L^{\ast}_{\bullet}(r(a))b-(-1)^{{|a|}{|b|}}L^{\ast}_{\bullet}(r(b))+a \cdot b). \ \ \ \ \ \forall \  a, b\in \mathcal{P}^{\ast}
\end{eqnarray}
The following result establishes a close relation between a post-Poisson superalgebra and a quasitriangular Poisson superbialgebra.
\end{thm}

\begin{thm} With the conditions and notations in Theorem \ref{quasitriang}, define four new even bilinear operations $[\cdot ,\cdot ], \diamond, \cdot, \succ : \mathcal{P}^{\ast}\otimes \mathcal{P}^{\ast} \longrightarrow \mathcal{P}^{\ast}$ on $\mathcal{P}^{\ast}$ as follows:
\begin{eqnarray}
[a, b]=-2ad_{\{,\}}^{\ast}(\beta(a))b, \ \ \ \ \ \ \  a\diamond b=ad_{\{,\}}^{\ast}(r(a))b,
\end{eqnarray}
\begin{eqnarray}
\ \ \ \ \ \ \ a \cdot b=2L^{\ast}_{\bullet}(\beta(a))b, \ \ \ \ \ \ \ \ \ \ \ \ \ a \succ b=-L^{\ast}_{\bullet}(r(a))b, \ \ \forall \  a, b\in \mathcal{P}^{\ast}.
\end{eqnarray}
Then $(\mathcal{P}^{\ast}, [\cdot ,\cdot ], \diamond, \cdot, \succ)$ becomes a post-Poisson superalgebra and $r$ is a homomorphism of Poisson
superalgebras from the associated Poisson superalgebra $\mathcal{P}(\mathcal{P}^{\ast})$ of $(\mathcal{P}^{\ast}, [\cdot ,\cdot ], \diamond, \cdot, \succ)$ to $(\mathcal{P}, \{,\}, \bullet)$.
\end{thm}

\end{document}